\documentclass{amsart}

\usepackage{amsrefs}
\usepackage[all]{xy}
\usepackage{syntonly}
\usepackage{enumerate}
\usepackage{hyperref}
\usepackage{amsfonts}
\usepackage{amssymb}
\usepackage{amsmath}
\usepackage{amsthm}
\usepackage[inline]{enumitem}
\usepackage{tikz}
\usepackage{graphics}
\usepackage{graphicx}
\usepackage{color}
\usepackage{comment}

\usepackage{lineno}

\usepackage{multirow}

\newtheorem{theorem}{Theorem}[section]
\newtheorem{corollary}[theorem]{Corollary}

\newtheorem{definition}[theorem]{Definition}
\newtheorem{lemma}[theorem]{Lemma}

\newtheorem{nonGlobalClaim}{Claim}[theorem]  

\newtheorem{observation}[theorem]{Observation}
\newtheorem{question}[theorem]{Question}

\newtheorem{fact}[theorem]{Fact}

\newtheorem{remark}[theorem]{Remark}

\specialcomment{com}{ \color{red} }{ \color{black}} 
\specialcomment{oldcom}{ \color{brown} }{\color{black}} 

\excludecomment{oldcom} 

\usepackage{stmaryrd}

\begin{document}
\title[Strongly proper forcing and some problems of Foreman]{Strongly proper forcing and some problems of Foreman}

\author{Sean Cox}
\email{scox9@vcu.edu}
\address{
Department of Mathematics and Applied Mathematics \\
Virginia Commonwealth University \\
1015 Floyd Avenue \\
Richmond, Virginia 23284, USA 
}

\author{Monroe Eskew}
\email{monroe.eskew@univie.ac.at}
\address{
Kurt G\"odel Research Center \\
University of Vienna \\
W\"ahringer Strasse 25 \\
1090 Wien \\
Austria
}

\thanks{The authors gratefully acknowledge support from the VCU Presidential Research Quest Fund.  The first author also acknowledges support from Simons Grant number 318467.}

\subjclass[2010]{03E05,03E35, 03E55, 03E57, 03E65
}

\begin{abstract}

We provide solutions to several problems of Foreman about ideals, several of which are closely related to Mitchell's notion of \emph{strongly proper} forcing.  We prove:
\begin{enumerate*}
  \item Presaturation of a normal ideal implies projective antichain catching, enabling us to provide a solution to a problem from  Foreman~\cite{MattHandbook} about ideal projections which is more comprehensive and simpler than the solution obtained in \cite{Cox_MALP}.
 \item We solve an older question from Foreman~\cite{MR819932} about the relationship between generic hugeness and generic almost hugeness. 
 \item Finally, we provide solutions to two technical questions from Foreman~\cite{MR3038554} and \cite{MattHandbook} related to his \emph{Duality Theorem}. 
\end{enumerate*}

\end{abstract}

\maketitle

\section{Introduction}

Foreman's \emph{Duality Theorem} is a powerful tool to characterize ideal quotients in many situations.  Mitchell's notion of \emph{strongly proper} forcing, a smaller class than Shelah's \emph{proper} forcing, has played an important role in recent years, especially because strongly proper forcings don't add new cofinal branches to ground model trees.  We answer several questions of Foreman that turn out to involve both the Duality Theorem and strongly proper forcing.

First we address the notions of \emph{antichain catching} and \emph{self-generic structures} that were introduced in Foreman-Magidor-Shelah~\cite{MR924672} and further utilized by Woodin~\cite{MR2723878} to prove deep facts about saturated and presaturated ideals.  Suppose $Z$ is a set and $\mathcal{I}$ is an ideal on the boolean algebra $\wp(Z)$;\footnote{For example, $Z = \kappa$ or $Z = \wp_\kappa(\lambda)$ for some cardinals $\kappa$, $\lambda$.} let $\mathbb{B}_{\mathcal{I}}$ denote the quotient $\wp(Z)/\mathcal{I}$.  Given a regular $\theta$ with $\mathbb{B}_{\mathcal{I}} \in H_\theta$ and an $M \prec (H_\theta, \in, \mathbb{B}_{\mathcal{I}})$, we say that $M$ is \textbf{$\boldsymbol{\mathcal{I}}$-self-generic}\footnote{This is sometimes also expressed by ``$M$ catches all of its $\mathcal{I}$ maximal antichains".} iff, letting $\sigma_M: H_M \to M \prec H_\theta$ be the inverse of the Mostowski collapse of $M$, the canonical $H_M$-ultrafilter\footnote{I.e.\ the collection $\{ A \in \sigma_M^{-1}\big( \wp(Z) \big)  \ : \  M \cap \text{sprt}(\mathcal{I}) \in \sigma_M(A)   \} $; see Section \ref{sec_Prelims} for details.} derived from $\sigma_M$
 is generic over $H_M$ for the forcing poset $\sigma_M^{-1}(\mathbb{B}_{\mathcal{I}})$.  Lemma 3.46 of \cite{MattHandbook} proves that the following are equivalent, for any normal ideal $\mathcal{I}$ that satisfies a common technical criterion:\footnote{Namely, that the cardinality of the support of $\mathcal{I}$ is the same as the cardinality of the universe of $\mathcal{I}$; see Section \ref{sec_Prelims} for the meaning of these terms.  In particular, this holds for any normal fine ideal on a regular uncountable cardinal.}
\begin{enumerate}
 \item $\mathcal{I}$ is saturated, i.e.\ $\mathbb{B}_{\mathcal{I}}$ has a certain chain condition (see Section \ref{sec_SatPresat}; if $Z = \kappa$ it is the $\kappa^+$-chain condition);
 \item\label{item_MeasureOneCondClub} The set of $\mathcal{I}$-self-generic submodels of $H_{2^{2^{|\text{trcl}(\mathcal{I})|}}}$ constitutes a measure one set in the \emph{conditional club filter relative to $\mathcal{I}$}.\footnote{See Section \ref{sec_Prelims}.}
\end{enumerate}

Cox-Zeman~\cite{Cox_MALP} introduced a property called \emph{projective antichain catching for $\mathcal{I}$}, which is the property obtained by weakening the ``measure one" in part \ref{item_MeasureOneCondClub} above to the weaker property of being \emph{$\mathcal{I}$-projective stationary}.  Projective antichain catching is, in turn, closely related to the following question from the \emph{Handbook of Set Theory}:
\begin{question}[Question 13 of Foreman~\cite{MattHandbook}]\label{q_Foreman}
Suppose that $\mathcal{J}$ is an ideal on $Z \subseteq \wp(\kappa^{+(n+1)})$, and $\mathcal{I}$ is the projected ideal on the projection of $Z$ to $Z' \subseteq \wp(\kappa^{+n})$.  Suppose that the canonical homomorphism from $\wp(Z')/\mathcal{I}$ to $\wp(Z)/\mathcal{J}$ is a regular embedding.\footnote{An embedding from one poset into another is \emph{regular} if it is order and incompatibility preserving, and pointwise images of maximal antichains remain maximal.}  Is $\mathcal{I}$ $\kappa^{+(n+1)}$ saturated?
\end{question}

Cox-Zeman~\cite{Cox_MALP} provided a negative answer to Question \ref{q_Foreman} in the case where $n=0$ and $\kappa$ is the successor of a regular cardinal.  In this paper we provide a more comprehensive negative solution to Question \ref{q_Foreman}, taking care of all $n \in \omega$, as well as the case where $\kappa$ is the successor of a singular cardinal, and with a simpler proof than \cite{Cox_MALP}.\footnote{There is a cost to the simpler proof, however.  Cox-Zeman~\cite{Cox_MALP} provided examples of ideals which had projective antichain catching but weren't even \emph{strong} ideals.  All ideals in this paper will be presaturated, and thus strong.}  First, we prove a general theorem about presaturated normal ideals:  
\begin{theorem}\label{thm_Presat_implies_Catch}
Suppose $\mathcal{I}$ is a normal, presaturated ideal.\footnote{The general definition of $\delta$-(pre)saturation appears in Definition \ref{def_SatPresat}.  For an $\mathcal{I}$ on some $Z' \subseteq \wp(\kappa^{+n})$ as in Question \ref{q_Foreman}, presaturated means $\kappa^{+(n+1)}$-presaturated, and in particular the quotient will be $\kappa^{+(n+1)}$-preserving.  }  Then $\mathcal{J}$ has projective antichain catching.
\end{theorem}

We will say that a triple $(n, \kappa, \mathcal{I})$ is a \textbf{counterexample for Question \ref{q_Foreman}} iff there exists some ideal $\mathcal{J}$ such that the tuple $(n,\kappa, \mathcal{I}$, $\mathcal{J})$ satisfies the assumptions of Question \ref{q_Foreman}, yet $\mathcal{I}$ is \textbf{not} $\kappa^{+(n+1)}$-saturated.   By the new Theorem \ref{thm_Presat_implies_Catch}, together with a result from \cite{Cox_MALP} (see Fact \ref{fact_ImpliesNegForeman} in Section \ref{sec_Prelims} below), in order to provide counterexamples for Question \ref{q_Foreman} it suffices to find presaturated, nonsaturated ideals.   If $\kappa$ is the successor of a regular cardinal, there is a highly modular way to do this:

\begin{theorem}\label{thm_SolveHBquestion}
Fix $n \in \omega$.  Suppose $\mu$ is regular, $\mu^{<\mu} = \mu$, $\kappa = \mu^+$, and $|\wp \wp (\kappa^{+n}) |< \kappa^{+ \omega}$.  Suppose $\mathcal{I}$ is a normal, $\kappa^{+(n+1)}$-presaturated ideal on $\wp_\kappa(\kappa^{+n})$ with uniform completeness $\kappa$ such that
\[
\Vdash_{\mathbb{B}_{\mathcal{I}}} \mathrm{cf}(\kappa) = \mu 
\]

Then there is a cardinal preserving poset $\mathbb{P}(\mu,\kappa)$ such that, letting $\bar{\mathcal{I}}$ denote the ideal generated by $\mathcal{I}$ in $V^{\mathbb{P}(\mu,\kappa)}$, there is some $S \in \bar{\mathcal{I}}^+$ such that $V^{\mathbb{P}(\mu,\kappa)} \models$  ``$(n,\kappa,\bar{\mathcal{I}} \restriction S)$ is a counterexample for Question \ref{q_Foreman}."    
\end{theorem}

The proof of Theorem \ref{thm_SolveHBquestion} goes through a technical theorem  which generalizes some arguments of Baumgartner-Taylor~\cite{MR654852} and makes use of the notion of \emph{strong properness} (see Theorem \ref{thm_KillSaturation} in Section \ref{sec_BT}).

Consider fixed $m, n \in \omega$ with $m \ge 1$.  Starting from a model with an almost huge cardinal, Kunen, Laver, and Magidor proved that one can force to obtain a model $V_{m,n}$ of ZFC such that GCH holds in $V_{m,n}$, and letting $\mu := (\omega_{m-1})^{V_{m,n}}$ and $\kappa:= (\omega_m)^{V_{m,n}}$, $V_{m,n} \models $ ``there is a $\kappa^{+(n+1)}$-saturated ideal on $\wp_{\kappa}(\kappa^{+n})$ whose quotient forces $\kappa$ to have cofinality $\mu$".\footnote{The case $m = 1$ and $n=0$ is due to Kunen~\cite{MR495118}; the case $m \ge 2$ and $n = 0$ is due to Laver (according to \cite{MattHandbook});  and the case $n > 0$ (and any $m \in \omega$) is due to Magidor; see Magidor~\cite{MR526312} and Theorem 7.43 of \cite{MattHandbook}. }   In particular, $V_{m,n}$ satisfies the assumptions of Theorem \ref{thm_SolveHBquestion} with $\kappa = \omega_m$, and Theorem \ref{thm_SolveHBquestion} yields the following corollary.   In what follows, we have used somewhat unusual notation $\omega_{m}^{+n}$ instead of $\omega_{m+n}$, in order to make comparison with Question \ref{q_Foreman} more obvious.  
\begin{corollary}\label{cor_Answer13_omega_m}
Fix $m,n \in \omega$ with $m \ge 1$, and let $V_{m,n}$ denote the model described in the previous paragraph.  Then there is a cardinal-preserving forcing extension of $V_{m,n}$ that satisfies:  ``there is a normal ideal $\mathcal{I}$ on $\wp_{\omega_m}(\omega_{m}^{+n})$ of completeness $\omega_m$ such that $(n,\omega_m,\mathcal{I})$ is a counterexample for Question \ref{q_Foreman}."
\end{corollary}

In the case where $\kappa$ is the successor of a singular cardinal, we do not have such a versatile result as Theorem \ref{thm_SolveHBquestion} which would allow us to produce counterexamples starting from an \emph{arbitrary} saturated or presaturated ideal.  Still, we are able to produce counterexamples in these cases as well, by a direct construction from large cardinals.  Again, we use the unusual notation $\aleph_{\omega+1}^{+n}$ instead of $\aleph_{\omega + 1 + n}$ in order to ease the comparison with Question \ref{q_Foreman}.
\begin{theorem}\label{thm_Singular}
Fix $n \in \omega$.  It is consistent, relative to a supercompact cardinal with an almost huge cardinal above it, that there is a normal ideal $\mathcal{I}$ on $\wp_{\aleph_{\omega+1}}(\aleph_{\omega+1}^{+n})$ of completeness $\aleph_{\omega+1}$ such that $\mathcal{I}$ is $\aleph_{\omega + 1}^{+(n+1)}$-presaturated, but not $\aleph_{\omega + 1}^{+(n+1)}$-saturated.   In this model, $(n,\aleph_{\omega+1}, \mathcal{I})$ is a counterexample for Question \ref{q_Foreman}. 
\end{theorem}

The methods used to prove Theorem \ref{thm_SolveHBquestion} also allow us to provide a negative answer the following older question:
\begin{question}[Foreman~\cite{MR819932}, page 11]\label{q_Foreman_Potent}
Suppose there is an $\aleph_2$-complete, precipitous ideal on $[\aleph_3]^{\aleph_2}$.  Is there an $\aleph_2$-complete, $\aleph_3$-saturated ideal on $\aleph_2$?
\end{question}

The answer is \emph{no}, even if we strengthen the assumption from precipitousness to a very strong form of presaturation: 
\begin{theorem}\label{thm_AnswerPotentAxioms}
It is consistent that there is an $\aleph_2$-complete, $\aleph_3$-presaturated ideal on $[\aleph_3]^{\aleph_2}$, yet there are no $\aleph_2$-complete, $\aleph_3$-saturated ideals on $\aleph_2$. 
\end{theorem}

Foreman's \emph{Duality Theorem} plays a central role in all of the results above.  In \cite{MR3038554} and \cite{MattHandbook}, Foreman proved several partial converses to the  Duality Theorem, and also discussed the following scenario.  Suppose $\mathcal{J}$ is an ideal in $V$ on some $Z \subseteq \wp(X)$, and $\mathbb{P}$ is a poset.  Let $G$ be $(V,\mathbb{P})$-generic and consider the ideal $\bar{\mathcal{J}}$ generated by $\mathcal{J}$ in $V[G]$.  Let $\bar{U}$ be $(V[G], \mathbb{B}_{\bar{\mathcal{J}}})$-generic, and let $U:= \bar{U} \cap \wp^V(Z)$. 
\begin{itemize}
 \item In what scenarios is $U$ generic over $V$ for $\mathbb{B}_{\mathcal{J}}$?  This is the apparent motivation for part \ref{item_RegQuestYes} of his Question \ref{q_Foreman_Regularity} below.
 \item In what scenarios are $G$ and $U$ mutually generic over $V$,  i.e.\ when is $G \times U$ generic over $V$ for the product $\mathbb{P} \times \mathbb{B}_{\mathcal{J}}$?  This is the apparent motivation for part \ref{item_RegQuestNo} of his Question \ref{q_Foreman_Regularity} below.
\end{itemize}

\begin{question}[Foreman]
\label{q_Foreman_Regularity}
Suppose that $\mathcal J$ is a normal ideal on $Z \subset \wp (X )$, that $\mathbb P$ is a $|X|^+$-cc partial ordering, and that $\bar{\mathcal J}$ (the ideal generated by $\mathcal{J}$) is an $|X |^+$-saturated normal ideal in $V^{\mathbb P}$. Is it true that there are $A,\dot S$ such that:
\begin{enumerate}[label=(\alph*)]
\item\label{item_RegQuestYes} (\cite{MR3038554}, page 337) the map $T \mapsto (1_{\mathbb{P}}, \check{T})$ 
 is a regular embedding from $\mathbb{B}_{\mathcal{J} \restriction A} \to \mathbb P * \dot{\mathbb{B}}_{\dot{\bar{\mathcal{J}}} \restriction \dot{S}}$?
 \item\label{item_RegQuestNo} (\cite{MattHandbook}, Open Question 44) the ``identity" map  $(p,T) \mapsto (p, \check{T})$ is a regular embedding from $\mathbb P \times \mathbb{B}_{\mathcal J \restriction A} \to \mathbb P * \dot{\mathbb{B}}_{\dot{\bar{\mathcal{J}}} \restriction \dot{S}}$?  
 \end{enumerate}
 \end{question}
 
We show that the answer is yes for part \ref{item_RegQuestYes} and no for part \ref{item_RegQuestNo}.  We also provide some exact characterizations involving various scenarios where part \ref{item_RegQuestNo} has a positive answer, which is tightly connected to strong properness; see Theorem \ref{thm_CC_Reg_StrProp}.

Theorem \ref{thm_Presat_implies_Catch} and the other results related to Question \ref{q_Foreman} are due to the first author.  Theorem \ref{thm_AnswerPotentAxioms}, i.e.\ the solution to Question \ref{q_Foreman_Potent}, was proved independently by both authors.  The answer to Question \ref{q_Foreman_Regularity} is due to the second author, and Theorem \ref{thm_CC_Reg_StrProp} is due to both authors.

The paper is organized as follows. Section \ref{sec_Prelims} provides the relevant background.   Section \ref{sec_PresatAntichainCatch} proves Theorem \ref{thm_Presat_implies_Catch}. Section \ref{sec_BT} proves Theorem \ref{thm_SolveHBquestion} and its corollaries. Section \ref{sec_Singular} proves Theorem \ref{thm_Singular}. Section \ref{sec_Potent} proves Theorem \ref{thm_AnswerPotentAxioms}.  Section \ref{sec_Regularity} answers Question \ref{q_Foreman_Regularity} and proves some related facts relating Question \ref{q_Foreman_Regularity} to strong properness.  Finally, Section \ref{sec_Questions} ends with some open questions.

\section{Preliminaries}\label{sec_Prelims}

 \subsection{Forcing preliminaries}\label{sec_ForcingPrelims}
 
An embedding $e: \mathbb{P} \to \mathbb{Q}$ between partial orders is called a \textbf{regular embedding} iff $e$ is order and incompatibility preserving, and if $A$ is a maximal antichain in $\mathbb{P}$ then $e[A]$ is maximal in $\mathbb{Q}$.  This is equivalent to requiring that for every $q \in \mathbb{Q}$, there is a (typically non-unique) $p_q \in \mathbb{P}$ such that whenever $p' \le p_q$ then $e(p')$ is compatible with $q$ in $\mathbb{Q}$.  Regularity of $e$ is also equivalent to the statement that 
\[
\Vdash_{\mathbb{Q}} \ \check{e}^{-1}[\dot{G}_{\mathbb{Q}}] \text{ is } (V,\mathbb{P}) \text{-generic}
\]

A closely related concept is that of \emph{strong properness}, introduced by Mitchell~\cite{MR2279659}.  Given a poset $\mathbb{P}$, a set $M$ such that $M \cap \mathbb{P}$ is a partial order and $\text{id}: M \cap \mathbb{P} \to \mathbb{P}$ is $\le_{\mathbb{P}}$ and $\perp_{\mathbb{P}}$-preserving,\footnote{The exact requirements on $M$ vary across the literature, but a sufficient condition on $M$ to ensure $\text{id}: M \cap \mathbb{P} \to \mathbb{P}$ has the required properties is for $(M,\in)$ to be elementary in $(H,\in)$ for some transitive structure $H$ such that $\mathbb{P} \in H$.} and a condition $p \in \mathbb{P}$, we say that \textbf{$\boldsymbol{p}$ is an $\boldsymbol{(M,\mathbb{P})}$ strong master condition} iff for every $p' \le p$, there is some (typically non-unique) $p'_M \in M \cap \mathbb{P}$ such that all extensions of $p'_M$ in $M \cap \mathbb{P}$ are compatible with $p'$.  This is sometimes expressed by saying ``$M \cap \mathbb{P}$ is regular in $\mathbb{P}$ below the condition $p$".  It is equivalent to requiring that 
\[
p \Vdash_{\mathbb{P}} \ \dot{G} \cap (M \cap \mathbb{P}) \text{ is generic over } V \text{ for the poset } M \cap \mathbb{P}
\]
We say that \textbf{$\boldsymbol{\mathbb{P}}$ is strongly proper with respect to the model $M$} if for all $p \in M \cap \mathbb{P}$ there exists an $(M,\mathbb{P})$-strong master condition below $p$.  

Note that all of these are really statements purely about $M \cap \mathbb{P}$, not really about $M$.  This has the following easy, but very useful, consequence:
\begin{lemma}\label{lem_DiffModelsSameInt}
Suppose $M$, $N$ are two sets such that $M \cap \mathbb{P} = N \cap \mathbb{P}$.  Then:
\begin{itemize}
 \item For every $p \in \mathbb{P}$, $p$ is an $(M,\mathbb{P})$ strong master condition if and only if $p$ is an $(N,\mathbb{P})$ strong master condition.
 \item $\mathbb{P}$ is strongly proper with respect to $M$ if and only if $\mathbb{P}$ is strongly proper with respect to $N$.
\end{itemize}
\end{lemma}

The following is another key feature of strong master conditions, which ordinary master conditions may lack:
\begin{observation}\label{obs_Sigma_0}
The statements ``$p$ is an $(M,\mathbb{P})$ strong master condition" and ``$\mathbb{P}$ is strongly proper with respect to $M$" are $\Sigma_0$ statements involving the parameters $p$, $\mathbb{P}$, and $M \cap \mathbb{P}$.
\end{observation}

\begin{corollary}\label{cor_Absolute}
Let $V \subset W$ be two transitive models of set theory.  Suppose $M, \mathbb{P} \in V$ and $p \in \mathbb{P}$.  Then $V$ and $W$ agree on the truth value of the statements ``$p$ is an $(M,\mathbb{P})$ strong master condition" and ``$\mathbb{P}$ is strongly proper with respect to $M$".
\end{corollary}
 
A nontrivial elementary embedding $j: V \to N$ is called \textbf{almost huge} iff $j$ is definable in $V$ and $N$ is closed under ${<}j(\text{crit}(j))$ sequences from $V$.  We will use the following theorem of Kunen~\cite{MR495118}, which was generalized by Laver for arbitrary regular $\mu$ (see \cite{MattHandbook}):
\begin{theorem}[Kunen, Laver]\label{thm_KunenLaver}
Suppose $j:V \to N$ is an almost huge embedding with critical point $\kappa$.  Let $\mu < \kappa$ be any regular cardinal and $n \in \omega$.  There is a $\mu$-directed closed, $\kappa$-cc poset $\mathbb{P} \subset V_\kappa$ and a regular embedding 
\[
e:  \mathbb{P}*\dot{\mathbb{S}} \to j(\mathbb{P})
\]
where $\dot{\mathbb{S}}$ is the $\mathbb{P}$-name for the $\kappa^{+n}$-closed Silver collapse that turns $j(\kappa)$ into $\kappa^{+n+1}$.  Moreover $e$ is the identity on $\mathbb{P}$, i.e.\ $e(p,1) = p$ for every $p \in \mathbb{P}$.
\end{theorem}

The ``Magidor variation" of the Kunen construction will be important for the proof of Theorem \ref{thm_Singular}.  If $F$ is a filter on a poset $\mathbb{Q}$, we write $\mathbb{Q}/F$ to denote the set of conditions in $\mathbb{Q}$ which are compatible with every member of $F$, with ordering inherited from $\mathbb{Q}$.
\begin{theorem}[Magidor~\cite{MR526312}; see also \cite{MattHandbook}]\label{thm_MagidorVariation}
Using the same assumptions and notations as Theorem \ref{thm_KunenLaver}, assume that $j$ is a \textbf{tower} embedding; that is, every element of $N$ is of the form $j(F)(j[\gamma])$ for some $\gamma < j(\kappa)$ and $F: \wp_\kappa(\gamma) \to V$.  Assume $G*H$ is $(V,\mathbb{P}*\dot{\mathbb{S}})$ generic.  Let $G'$ be $j(\mathbb{P})/e[G*H]$-generic over $V[G*H]$.  Then the map $j$ can be lifted to $j_{G'}:V[G] \to N[G']$, and if $H'$ is generic over $V[G']$ for the poset $j_{G'}(\mathbb{S})/j_{G'}[H]$ then $H'$ is generic over $N[G']$ for the poset $j_{G'}(\mathbb{S})$, and $j_{G'}$ can be further lifted to an elementary embedding
\[
j_{G'*H'}: V[G*H] \to N[G'*H']
\]
\end{theorem}

\subsection{Presaturated forcings}

The notion of a $\delta$-presaturated poset was, as far as the authors are aware, introduced in Baumgartner-Taylor~\cite{MR654852}.

\begin{definition}\label{def_PresatPoset}
We say that a poset $\mathbb{P}$ is \textbf{$\boldsymbol{\delta}$-presaturated} iff whenever $\mathcal{A}$ is a ${<}\delta$-sized collection of antichains, there are densely many $p \in \mathbb{P}$ such that for every $A \in \mathcal{A}$, $p$ is compatible with (strictly) fewer than $\delta$ many members of $A$.   
\end{definition}

\begin{remark}
Replacing the word ``antichains" with ``maximal antichains" in Definition \ref{def_PresatPoset} results in an equivalent definition.
\end{remark}

If $\delta$ is a regular uncountable cardinal and $H$ is a superset of $\delta$, then $\wp_\delta(H)$ denotes all subsets of $H$ of size $<\delta$, and $\wp^*_\delta(H)$ denotes the set of $z \in \wp_\delta(H)$ such that $z \cap \delta \in \delta$.  For $\delta = \omega_1$, $\wp^*_\delta(H)$ and $\wp_\delta(H)$ are essentially the same set (modulo clubs), but for $\delta > \omega_1$ they can differ, e.g.\ in the presence of Chang's Conjecture; see \cite{MattHandbook}.   The reason we work with $\wp^*_\delta(H)$ rather than $\wp_\delta(H)$ will become apparent later, in the proof of Fact \ref{fact_ProperImpliesPresat}.

We will often use the notion of being ``proper on a stationary set" (defined below), which implies presaturation.  Given a poset $\mathbb{P}$, an $M \prec (H_\theta,\in,\mathbb{P})$ for some $\theta$ such that $\mathbb{P} \in H_\theta$, and a condition $p \in \mathbb{P}$, recall that \textbf{$\boldsymbol{p}$ is an $\boldsymbol{(M,\mathbb{P})}$-master condition} iff for every dense $D \in M$, $D \cap M$ is predense below $p$; this is equivalent to the assertion $p \Vdash M[\dot{G}_{\mathbb{P}}] \cap V = M$.  We say that \textbf{$\boldsymbol{\mathbb{P}}$ is proper with respect to $\boldsymbol{M}$} iff for every $p \in M \cap \mathbb{P}$ there exists a $p' \le p$ such that $p'$ is an $(M,\mathbb{P})$-master condition.  If $S$ is a stationary collection of $M$ such that $\mathbb{P}$ is proper with respect to $M$ for every $M \in S$, we say that \textbf{$\boldsymbol{\mathbb{P}}$ is proper on $\boldsymbol{S}$.}  Finally, we say that \textbf{$\mathbb{P}$ is $\boldsymbol{\delta}$-proper on a stationary set} iff it is proper on $S$ for some stationary set $S$.

\begin{fact}\label{fact_ProperImpliesPresat}
If $\mathbb{P}$ is $\delta$-proper on a stationary set, then it is $\delta$-presaturated.
\end{fact}
\begin{proof}
We actually prove that a weaker property than being $\delta$-proper on a stationary set suffices.  Namely, suppose for every $p \in \mathbb{P}$ there are stationarily many $N \in \wp^*_\delta(H_{\theta(\mathbb{P})})$ such that some $p' \le p$ is an $(N,\mathbb{P})$-master condition.  This can be viewed as a ``non-diagonal" version of being $\delta$-proper on a stationary set,\footnote{The latter is ``diagonal" because it requires being able to extend \emph{every condition in the model} to a master condition for that model.}  and  is equivalent, by Proposition 3.2 of Foreman-Magidor~\cite{MR1359154}, to saying that
\[
\Vdash_{\mathbb{P}} \ \forall \theta \ge \delta \ \ \big( \wp^*_\delta(\theta) \big)^V  \text{ is stationary}
\]
For $\delta = \omega_1$, Foreman-Magidor~\cite{MR1359154} call such forcings \emph{reasonable}.

Assume $p \in \mathbb{P}$, $\lambda < \delta$, and $\vec{A} = \langle A_i \ : \ i < \lambda \rangle$ is a sequence of antichains in $\mathbb{P}$.  By assumption there is an 
\[
N \prec (H_\theta, \in, p, \mathbb{P}, \vec{A} )
\]
such that $N \in \wp^*_\delta(H_\theta)$ and there is some $p' \le p$ which is an $(N,\mathbb{P})$-master condition.  We prove that
\begin{equation}
\forall i < \lambda \ \ \{ r \in A_i \ : \ r \parallel p' \} \subset N
\end{equation}
Since $|N|<\delta$ this will complete the proof.  Fix $i < \lambda$; since $N \cap \delta \in \delta$ and $\lambda \in N \cap \delta$ then $\lambda \subset N$; so in particular $i \in N$ and thus $A_i \in N$.\footnote{This is why we work with $\wp^*_\delta(H_\theta)$ instead of $\wp_\delta(H_\theta)$.}  Extend $A_i$ to some $\bar{A}_i \in N$ which is a maximal antichain.    Suppose for a contradiction that there is some $r \in A_i \setminus N$ which is compatible with $p'$.  Then $r \in \bar{A}_i$.  Let $G$ be generic with $r$ and $p'$ both in $G$.    Then $r$ is the unique member of $G \cap \bar{A}_i$.   Since $p'$ is an $(N,\mathbb{P})$-master condition and $\bar{A}_i \in N$ then $G \cap N \cap \bar{A}_i \ne \emptyset$; but $r$ is not in $N$ and this contradicts that $r$ was the unique member of $G \cap \bar{A}_i$.
\end{proof}

Whether $\delta$-presaturation is even closed under 2-step iterations is apparently open (see Question \ref{q_2stepPresat}).  Parts \ref{item_FiniteIteration} and \ref{item_2Step} of the following folklore fact collects several special cases of iterations of $\delta$-presaturated posets that behave nicely, and which suffice for our applications:
\begin{fact}\label{fact_IterationsDeltaProper}
Suppose $\delta$ is regular and uncountable.  Then:
\begin{enumerate}
 \item\label{item_DeltaCC}  If $\mathbb{P}$ is $\delta$-cc and $M \prec (H_\theta,\in,\mathbb{P})$ is such that $M \cap \delta \in \delta$, then $1_{\mathbb{P}}$ is an $(M,\mathbb{P})$-master condition.
  \item\label{item_DeltaClosed} Every $\delta$-closed forcing is $\delta$-proper on a stationary set, namely the stationary set $\text{IA}_{<\delta}$ of $M \in \wp^*_\delta(H_\theta)$ such that $|M| = |M \cap \delta|$ and $M$ is internally approachable of length $<\delta$.\footnote{That is, there is some $\zeta < \delta$ and some $\subset$-increasing and continuous sequence $\langle N_i \ : \ i < \zeta \rangle$ whose union is $M$, and such that $\vec{N} \restriction i \in M$ for every $i < \zeta$.} 
 \item\label{item_FiniteIteration} Any finite iteration of $\delta$-cc and $\delta$-closed posets---i.e.\ any iteration of the form 
 \[\mathbb{Q}_0 * \dot{\mathbb{Q}}_1 * \dots * \dot{\mathbb{Q}}_n,
 \]
  where $\Vdash_{\mathbb{Q}_0 * \dot{\mathbb{Q}}_1 * \dots * \dot{\mathbb{Q}}_{i-1}}$ ``$\dot{\mathbb{Q}}_i$ is either $\delta$-cc or $\delta$-closed" for every $i \le n$---is $\delta$-proper on $\text{IA}_{<\delta}$ and hence (by Fact \ref{fact_ProperImpliesPresat}) $\delta$-presaturated.
 \item\label{item_2Step} Any two-step iterations of the form ``$\delta$-proper on a stationary set, followed by $\delta$-cc" is $\delta$-proper on a stationary set and hence (by Fact \ref{fact_ProperImpliesPresat}) $\delta$-presaturated.
 \end{enumerate}
\end{fact}


\begin{proof}
To see part \ref{item_DeltaCC}:  given any $A \in M$ that is a maximal antichain in $\mathbb{P}$, the fact that $|A|<\delta$ and $M \cap \delta \in \delta$ implies $A \subset M$.  It follows easily that $1_{\mathbb{P}} \Vdash M[\dot{G}] \cap V = M$, which implies $1_{\mathbb{P}}$ is a master condition for $M$.

Part \ref{item_DeltaClosed} is due to Foreman-Magidor~\cite{MR1359154}.

Before proving part \ref{item_FiniteIteration}, first note the following routine fact, which we leave to the reader: 
\begin{enumerate}[label=(\Alph*)]
 \item\label{item_PreserveIA} For any poset $\mathbb{P}$ and any $M \prec (H_\theta,\in, \mathbb{P})$, if $\langle N_i \ : \ i < \zeta \rangle$ witnesses that $M \in \text{IA}_{<\delta}$, and $G$ is $(V,\mathbb{P})$-generic, then in $V[G]$, $\langle N_i[G] \ : \ i < \zeta \rangle$ witnesses that $M[G] \in \text{IA}_{<\delta}$. 
\end{enumerate}
Now return to part \ref{item_FiniteIteration}.  Suppose $M \in \text{IA}_{<\delta}$ is such that $M \prec (H_\theta,\in,\vec{\mathbb{Q}})$.  Using \ref{item_PreserveIA}, together with the already-proven parts \ref{item_DeltaCC} and \ref{item_DeltaClosed} of the current fact, it is routine to show by induction on $i$ that each successive $\dot{\mathbb{Q}}_i$ is forced by $\vec{\mathbb{Q}} \restriction i$ to be proper with respect to $M[\dot{G}_{<i}]$, where $\dot{G}_{<i}$ is the name for the $\vec{\mathbb{Q}} \restriction i$-generic object.  It then follows that $\vec{\mathbb{Q}} \restriction (i+1)$ is proper with respect to $M$.

To see part \ref{item_2Step}, suppose $\mathbb{P}$ is $\delta$-proper on a stationary set, and $\dot{\mathbb{Q}}$ is a $\mathbb{P}$-name for a $\delta$-cc poset.  Then there is a regular $\theta$ with $\mathbb{P}*\dot{\mathbb{Q}} \in H_\theta$ such that for stationarily many $M \in \wp^*_\delta(H_\theta)$, $\mathbb{P}$ is proper with respect to $M$.  Fix any such $M$; we show that $\mathbb{P} * \dot{\mathbb{Q}}$ is proper with respect to $M$.  Given any $(p,\dot{q}) \in M \cap \mathbb{P}*\dot{\mathbb{Q}}$, fix a $p' \le p$ that is an $(M,\mathbb{P})$-master condition.  Part \ref{item_DeltaCC} of the current fact then implies that
\[
p' \Vdash_{\mathbb{P}} \ 1_{\dot{\mathbb{Q}}} \text{ is an } (M[\dot{G}_{\mathbb{P}}], \dot{\mathbb{Q}}) \text{-master condition (and hence so is } \dot{q} \text{)}
\]
Then $(p',\dot{q})$ is an $(M,\mathbb{P}*\dot{\mathbb{Q}})$-master condition.
\end{proof}

We will use some other standard facts about presaturation.

\begin{fact}\label{fact_DeltaPresatPreserveCof}
If $\delta$ is regular and $\mathbb{P}$ is $\delta$-presaturated, then 
\[
\Vdash_{\mathbb{P}} \ \mathrm{cof}^V(\ge \delta) = \mathrm{cof}^{V[\dot{G}]}(\ge \delta)
\]
\end{fact}

The following is a partial converse to Fact \ref{fact_DeltaPresatPreserveCof}:
\begin{fact}[Baumgartner-Taylor~\cite{MR654852}]\label{fact_DeltaPlusOmega}
Suppose $\delta$ is a regular uncountable cardinal, $\mathbb{P}$ is $\delta^{+\omega}$-cc, and 
\[
\forall n \in \omega \ \ \Vdash_{\mathbb{P}} \mathrm{cf}\big( \delta^{+n} \big)^V \ge \delta
\]
Then $\mathbb{P}$ is $\delta$-presaturated. 
\end{fact}

Often---especially in applications of Foreman's Duality Theorem (Section \ref{sec_Duality})---one encounters some 2-step iteration $\mathbb{P} * \dot{\mathbb{Q}}$ which has some nice property, such as being proper with respect to some models in $\wp^*_\delta(H_\theta)$, and would like to know that $\mathbb{P}$ forces that $\dot{\mathbb{Q}}$ has a similar property.\footnote{This occurs, e.g.\ in the proof of part \ref{item_PresPresatKillSat} of Theorem \ref{thm_KillSaturation}.}  Although this works for some natural properties like the $\delta$-cc, it can fail for others.  For example, Cummings~\cite{MR2768691} gives an example of a proper 2-step iteration $\mathbb{P}*\dot{\mathbb{Q}}$, such that $\mathbb{P}$ forces that $\dot{\mathbb{Q}}$ is \textbf{not} proper.  However, the quotients of presaturated posets are well-behaved.  If $q$ is a condition in a poset $\mathbb{P}$ and $\vec{A}$ is a sequence of antichains in $\mathbb{P}$, let us make the ad-hoc definition that \textbf{$\boldsymbol{q}$ is $\boldsymbol{\delta}$-good for $\vec{A}$} iff for every $\alpha \in \text{dom}(\vec{A})$, $q$ is compatible with strictly fewer than $\delta$ many members of $A_\alpha$.  So $\delta$-presaturation is equivalent to requiring that whenever $\vec{A}$ is a ${<}\delta$-length sequence of antichains, there are densely many conditions which are $\delta$-good for $\vec{A}$.
\begin{lemma}\label{lem_MonroeQuotient}
If $\delta$ is regular and $\mathbb{P}*\dot{\mathbb{Q}}$ is $\delta$-presaturated, then:
\begin{itemize}
 \item $\mathbb{P}$ is $\delta$-presaturated; and
 \item $\mathbb{P}$ forces that $\dot{\mathbb{Q}}$ is $\delta$-presaturated.
\end{itemize}
\end{lemma}
\begin{proof}
That $\mathbb{P}$ is $\delta$-presaturated is routine, and we leave it to the interested reader.  The main content of the lemma is proving the second bullet.

Suppose $\lambda < \delta$ and $p_0 \in \mathbb{P}$ forces that $\dot{\vec{A}} = \langle \dot{A}_\alpha \ :  \alpha < \check{\lambda} \rangle$ is a sequence of antichains in $\dot{\mathbb{Q}}$, that $\dot{q}$ is a condition in $\dot{\mathbb{Q}}$, and that for each $\alpha < \lambda$, $\langle \dot{q}_\alpha^\xi \ : \ \xi < \dot{\rho}_\alpha \rangle$ is an enumeration of $\dot{A}_\alpha$.   We need to find some $p \le p_0$ which forces that some condition below $\dot{q}$ is $\delta$-good for $\dot{\vec{A}}$.  

For $\alpha < \lambda$ let $B_\alpha$ be a maximal antichain below $p_0$ of conditions $p$ that decide the value of $\dot{\rho}_\alpha$ as some $\rho_{\alpha,p}$, and let
\[
C_\alpha:= \{ (p, \dot{q}_\alpha^\xi) \ : \ p \in B_\alpha \text{ and } \xi < \rho_{\alpha,p} \}
\]
Clearly each $C_\alpha$ is an antichain.  By $\delta$-presaturation of $\mathbb{P}*\dot{\mathbb{Q}}$ there is some $(p_1, \dot{q}') \le (p_0, \dot{q})$  such that
\begin{equation}\label{eq_GoodForCalpha}
 (p_1, \dot{q}') \text{ is } \delta \text{-good for } \vec{C}=\langle C_\alpha \ : \ \alpha < \lambda \rangle
\end{equation}

The following claim will complete the proof:
\begin{nonGlobalClaim}\label{clm_p1_Forces}
$p_1 \Vdash$ $\dot{q}'$ is $\delta$-good for $\dot{\vec{A}}$. 
\end{nonGlobalClaim}
Suppose not.  Let $p_2 \le p_1$ decide some counterexample $\alpha$, i.e.\ $p_2$ forces that $\dot{q}'$ is compatible with at least $\delta$ many members of $\dot{A}_{\check{\alpha}}$.  Since $B_\alpha$ was a maximal antichain below $p_0$ there is some $p \in B_\alpha$ which is compatible with $p_2$; let $p_3$ witness this compatibility.  Then
\begin{equation}\label{eq_ForcesBigSet}
p_3 \Vdash | \{ \xi < \rho_{\alpha,p} \ : \   \dot{q}' \parallel \dot{q}_\alpha^\xi  \}| \ge \delta
\end{equation}

On the other hand, 
\begin{equation}\label{eq_SmallSet}
| \rho_{\alpha,p} \setminus \{ \xi \ : \   p_3 \Vdash \ \dot{q}' \perp \dot{q}_\alpha^\xi  \}|<\delta
\end{equation}
because otherwise $(p_3, \dot{q}')$ would be compatible with at least $\delta$-many members of $C_\alpha$, contradicting \eqref{eq_GoodForCalpha}.    Let $G$ be $\mathbb{P}$-generic with $p_3 \in G$, let $q':= (\dot{q}')_G$, and $q^\xi_\alpha:= (\dot{q}^\xi_\alpha)_G$ for each $\xi < \rho_{\alpha,p}$.  In $V[G]$ let 
\[
Z:= \{ \xi < \rho_{\alpha,p} \ : \ q' \parallel q^\xi_\alpha \}
\]

Since $\mathbb{P}*\dot{\mathbb{Q}}$ is $\delta$-presaturated by assumption, then in particular by Fact \ref{fact_DeltaPresatPreserveCof} it preserves regularity of $\delta$, and hence so does $\mathbb{P}$.\footnote{Actually $\mathbb{P}$ is $\delta$-presaturated too, but we're not explicitly using that here.}  So $\delta$ is still a regular cardinal in $V[G]$.  Since $\delta$ is still a cardinal in $V[G]$ then  \eqref{eq_ForcesBigSet} and \eqref{eq_SmallSet} together imply
\[
V[G] \models \ |Z|<\delta=|\delta| \le |Z|
\]
which is a contradiction.  This completes the proof of the claim and the lemma.
\end{proof}

\subsection{Saturation and presaturation of ideals}\label{sec_SatPresat}

Following the convention of Foreman~\cite{MattHandbook} and elsewhere, we say that $\boldsymbol{\mathcal{I}}$ \textbf{ is an ideal on $\boldsymbol{Z}$} to really mean that $\mathcal{I}$ is an ideal on the boolean algebra $\wp(Z)$.  As pointed out on page 893 of \cite{MattHandbook}, this is an abuse of terminology, but we stick with it in order to be consistent with the original wording of Questions \ref{q_Foreman} and \ref{q_Foreman_Potent}.  We will exclusively work with \textbf{fine} ideals, because it simplifies several issues involving the baggage that ideals carry.  We say that \textbf{$\boldsymbol{\mathcal{I}}$ is a fine ideal on $\boldsymbol{Z}$} if there exists a set $X$ such that $Z \subseteq \wp(X)$, $Z \ne \emptyset$, $\mathcal{I}$ is an ideal on the boolean algebra $\wp(Z)$,\footnote{i.e.\ $\emptyset \in \mathcal{I}$, $Z \notin \mathcal{I}$, and if $A,B \in \wp(Z)$ are both in $\mathcal{I}$ then so is their union.} and for every $x \in X$ the following set is in $\mathcal{I}$:
\[
Z \setminus \{ z \in Z \ : \ x \in z   \}
\]

The intended set $Z$ must be specified, but will often be called ``the" universe of $\mathcal{I}$ and denoted $\textbf{univ}\boldsymbol{(\mathcal{I})}$. If $\mathcal{I}$ is a fine ideal on $Z$, then the set $X$ as above is computable from $\mathcal{I}$ and $Z$ as follows.  Let $\text{Dual}_Z(\mathcal{I})$ denote the dual filter of $\mathcal{I}$ in $\wp(Z)$, i.e.\ all sets of the form $Z \setminus A $ where $A \in \mathcal{I}$.  Then 
\begin{equation}\label{eq_WhatIsX}
X = \bigcup \bigcup \text{Dual}_Z(\mathcal{I})
\end{equation}
If $\mathcal{I}$ is a fine ideal on $Z$ then the set \eqref{eq_WhatIsX} will be called the \textbf{support of $\boldsymbol{\mathcal{I}}$}, and denoted by $\textbf{sprt}\boldsymbol{(\mathcal{I})}$.  The support of an ideal provides the indexing set for  \textbf{diagonal intersections} and \textbf{diagonal unions}; e.g.\ the diagonal union of a sequence $\langle A_x \ : \ x \in X\rangle $ of elements of $\mathcal{I}$ is the set
\[
\nabla_{x \in X} A_x := \{ M \in Z \ : \  \exists x \in X \cap M \ M \in A_x  \}
\]
A fine ideal is called \textbf{normal} iff it is closed under diagonal unions.  Some standard examples of normal, fine ideals are:
\begin{itemize}
 \item The set of all nonstationary subsets of a regular uncountable cardinal $\kappa$ is a normal ideal on $\kappa$, where $Z = X = \kappa$ in the notation from above.
 \item For regular uncountable $\kappa$, the set of all nonstationary subsets of $\wp_\kappa(\lambda)$ is a normal ideal on $\wp_\kappa(\lambda)$, where $Z =  \wp_\kappa(\lambda)$ and $X = \lambda$ in the notation above.
\end{itemize}

If $\mathcal{I}$ is a fine ideal, then its support---i.e.\ $\text{sprt}(\mathcal{I})= \bigcup S$ where $S$ is any member of $\mathcal{I}^+$---is the relevant indexing set for taking diagonal unions and intersections.   It follows that if $\mathcal{I}$ is normal and fine, and $\mathcal{A}$ is an antichain of size at most $|\text{sprt}(\mathcal{I})|$, then $\mathcal{A}$ can be \emph{disjointed}; i.e.\ there is a collection $A$ of pairwise disjoint $\mathcal{I}$-positive sets such that $\mathcal{A} = \{ [S]_{\mathcal{I}} \ : \ S \in A \}$ (see \cite{MattHandbook}).  This makes maximal antichains of size $\le |\text{sprt}(\mathcal{I})|$ easier to work with; \cite{MattHandbook} has extensive information about disjointing antichains.  The following general definitions---which make sense for any fine ideal---agree with all special cases in the literature that the authors are aware of.

\begin{definition}\label{def_SatPresat}
Let $\mathcal{I}$ be a fine ideal on a set $Z$.  We say that $\mathcal{I}$ is:
\begin{itemize}
 \item \textbf{$\boldsymbol{\theta}$-saturated} iff $\wp(Z)/\mathcal{I}$ has the $\theta$-cc.  
 \item \textbf{saturated} iff $\mathcal{I}$ is $|\text{sprt}(\mathcal{I})|^+$-saturated. 

 \item \textbf{$\boldsymbol{\theta}$-presaturated} iff $\wp(Z)/\mathcal{I}$ is a $\theta$-presaturated poset, as in Definition \ref{def_PresatPoset}.

 \item \textbf{presaturated} iff $\mathcal{I}$ is $| \text{sprt}(\mathcal{I}) |^+$-presaturated.
 \end{itemize}
\end{definition}  
Clearly saturation implies presaturation.  If $\kappa$ is a successor cardinal, $\kappa \le \lambda$, and $\mathcal{I}$ is an ideal on $\wp_\kappa(\lambda)$ whose dual concentrates on 
\begin{equation}\label{eq_SupercompactSet}
\{ x \in \wp_\kappa(\lambda) \ : \  x \cap \kappa \in \kappa  \}
\end{equation}
then generic ultrapowers by $\mathcal{I}$ always collapse all cardinals in the interval $[\kappa,\lambda]$, so $\lambda^+$-saturation/presaturation is the best possible in this scenario.\footnote{And is consistent, by the ``skipping cardinals" technique of Magidor described in \cite{MattHandbook}.}  If the set \eqref{eq_SupercompactSet} has measure one for the ideal, we will sometimes say that the ideal is of \textbf{supercompact type}.   In short, if $\mathcal{I}$ is supercompact-type and its completeness is a successor cardinal, then $|\text{sprt}(\mathcal{I})|^+$ is the best degree of saturation/presaturation possible.  Non-supercompact-type ideals can sometimes have better saturation properties.  For example, Magidor showed how to construct, for any given $n \ge 1$, an $\omega_{n+2}$-saturated, $\omega_n$-complete, normal ideal on $[\omega_{n+2}]^{\omega_n}$; that is, an ideal whose saturation degree is the same as the cardinality of its support.\footnote{See the ``skipping cardinals" technique described in \cite{MattHandbook}.  It is open whether the $+2$ in Magidor's theorem can be replaced by $+1$; see Question \ref{q_SaturatedHuge}.}

The following is standard, but because it is central to the proof of Theorem \ref{thm_Presat_implies_Catch}, we briefly sketch the proof for the convenience of the reader.

\begin{lemma}\label{lem_PresatImpliesClosure}[\cite{MattHandbook}]
Suppose $\mathcal{I}$ is a normal fine presaturated ideal with support of size $\lambda$.  Then generic ultrapowers by $\mathcal{I}$ are always closed under $\lambda$-sequences from the generic extension.
\end{lemma}  
\begin{proof}
(Sketch) Without loss of generality, $\text{sprt}(\mathcal{I}) = \lambda$.  Suppose $S \in \mathcal{I}^+$ and $S$ forces that $\dot{F}: \lambda \to \text{ult}(V,\dot{G})$.  For each $i < \lambda$, let $\mathcal{A}_i$ be an antichain in $\mathcal{I}^+$ which is maximal below $S$, and every condition in $\mathcal{A}_i$ decides the function which represents $\dot{F}(i)$.  So for each $T \in \mathcal{A}_i$ there is a $f_{i,T} \in V$ such that $T \Vdash \dot{F}(i) = [\check{f}_{i,T}]_{\dot{G}} = j_{\dot{G}}(\check{f}_{i,T})\big( j_{\dot{G}}[\lambda] \big)$.  Since $\mathcal{I}$ is presaturated, there is an $S' \le S$ such that for each $i < \lambda$, $S'$ is compatible with at most $|\text{sprt}(\mathcal{I})|=\lambda$-many members of $\mathcal{A}_i$.  By pointwise intersecting such members of $\mathcal{A}_i$ with $S'$, we can without loss of generality assume that each $\mathcal{A}_i$ is a maximal antichain below $S'$, and $|\mathcal{A}_i |\le \lambda$.  Normality and fineness of $\mathcal{I}$ ensures\footnote{By Proposition 2.23 of \cite{MattHandbook}.} that any antichain of size $\lambda = \text{sprt}(\mathcal{I})$ can, without loss of generality, be assumed to be pairwise disjoint; say $\mathcal{A}_i = \{ T^i_\zeta \ : \ \zeta < \lambda \}$ and $T^i_\zeta \cap T^i_{\zeta'} = \emptyset$ whenever $\zeta \ne \zeta'$.  Then one may define $h_i: \text{univ}(\mathcal{I}) \to V$ such that $h_i \restriction T^i_\zeta = f_{i,T^i_\zeta} \restriction T^i_\zeta$ for every $\zeta < \lambda$.  It is then straightforward to check that $S'$ forces $j_{\dot{G}}(h_i)\big( j_{\dot{G}} [\lambda] \big)  = \dot{F}(i)$ for every $i < \lambda$ and that $X \mapsto \langle h_i(X) \ : \ i < \lambda  \rangle$ represents $\dot{F}$ in the generic ultrapower.
\end{proof}  

If $j: V \to N$ is a generic ultrapower by a normal fine ideal with support $H$, then $[\text{id}]_G = j_{G}[H]$, so another way to state the conclusion of Lemma \ref{lem_PresatImpliesClosure} is that the generic ultrapower is closed under $|[\text{id}]_G|$-length sequences from $V[G]$.


\begin{corollary}\label{cor_KappaSuccessorCard}
Assume $\mathcal{I}$ is a normal presaturated ideal with uniform completeness $\kappa$, where $\kappa$ is a successor cardinal.  If $G$ is $(V,\mathbb{B}_{\mathcal{I}})$-generic and $j:V \to_G N$ is the generic ultrapower, $N$ is closed under ${<}j(\kappa)$ sequences from $V[G]$. 
\end{corollary}  
\begin{proof}
Let $\lambda:= |\text{sprt}(\mathcal{I})|$, and $j: V \to_G N$ be as in the hypotheses of the corollary.  Presaturation of $\mathcal{I}$ implies
\begin{equation}\label{eq_LambdaPlusStillCard}
\lambda^{+V} \text{ is a cardinal in } V[G]
\end{equation}
Let $\mu$ be the cardinal predecessor of $\kappa$ in $V$.  Then $j(\kappa) = \mu^{+N}$, and together with \eqref{eq_LambdaPlusStillCard} this implies that $j(\kappa) \le \lambda^+$.  Then any $<j(\kappa)$-length sequence of elements of $N$ in $V[G]$ can be indexed by $\lambda$, and it follows from Lemma \ref{lem_PresatImpliesClosure} that such a sequence is in $N$.   
\end{proof}

We remark that the assumption that $\kappa$ is a successor cardinal is necessary for Corollary \ref{cor_KappaSuccessorCard}.  For example, if $\mathcal U$ is a normal measure on a cardinal $\kappa$ and $\mathcal I$ is the dual ideal of $\mathcal U$, then $\mathcal I$ is 2-saturated, and the forcing $\mathbb B_{\mathcal I}$ is trivial.  If $j : V \to N$ is the ultrapower embedding, then it is well-known that $2^\kappa < j(\kappa) < (2^\kappa)^+$.  Therefore, $N$ is not closed under $2^\kappa$-sequences since it believes $j(\kappa)$ is a cardinal.


The theory leading to Corollary \ref{cor_KappaSuccessorCard}, in particular the ``disjointing" of antichains, seem to require normality of the ideal if the support of the ideal is larger than the completeness.  For $\kappa$-complete, uniform ideal on $\kappa$, however, the normality assumption is not necessary:
\begin{fact}[Propositions 2.9 and 2.14 of \cite{MattHandbook}]\label{fact_SatNonNormal}
Suppose $\mathcal{I}$ is a saturated (i.e.\ $\kappa^+$-saturated), $\kappa$-complete uniform ideal on $\kappa$.  Then if $U$ is generic for $\wp(\kappa)/\mathcal{I}$ and $j: V \to N$ is the generic ultrapower, then $N$ is always closed under $\kappa$ sequences from $V[U]$.
\end{fact}


\subsection{Antichain catching and ideal projections}

The following concept was isolated in Foreman-Magidor-Shelah~\cite{MR924672}:
\begin{definition}
Let $\mathcal{I}$ be a normal fine ideal in $H_\theta$, and $M \subset H_\theta$ with $\mathcal{I} \in M$.  If $\mathcal{A} \in M$ is a maximal antichain in $\mathbb{B}_{\mathcal{I}}$, we say that \textbf{$\boldsymbol{M}$ catches $\boldsymbol{\mathcal{A}}$} iff there is some $S \in \mathcal{A} \cap M$ such that $M \cap \text{sprt}(\mathcal{I}) \in S$.  We say that \textbf{$\boldsymbol{M}$ is $\boldsymbol{\mathcal{I}}$-self-generic} iff $M$ catches $\mathcal{A}$ for every maximal antichain $\mathcal{A} \in M$. For regular $\theta$ such that $\mathcal{I} \in H_\theta$, $S^{\theta}_{\mathcal{I}\text{-self-generic}}$ denotes the set of $\mathcal{I}$-self-generic subsets of $H_\theta$.
\end{definition}
Note that we do not place any constraint on the cardinality of $\mathcal{I}$-self-generic structures.  We also do not require that $M$ is elementary in $(H_\theta,\in)$, as this would be overly restrictive in the proof of Theorem \ref{thm_Presat_implies_Catch}.  The following fact is easy:
\begin{fact}\label{fact_SelfGen_TFAE}
Let $\mathcal{I}$ be a normal fine ideal, and $M$ be a set such $\mathcal{I} \in M$, $(M, \in)$ is extensional, and $(M,\in) \models \text{ZF}^-$.   The following are equivalent:
\begin{enumerate}
 \item $M$ is $\mathcal{I}$-self-generic.
 \item If $\sigma_M: H_M \to M$ is the inverse of the Mostowski collapse of $M$, then the following object is a $\big(H_M, \sigma_M^{-1}(\mathbb{B}_{\mathcal{I}}\big)$-generic filter:
 \[
\mathcal{U}(M,\mathcal{I}):= \big\{ X \in  \sigma_M^{-1}\big( \text{univ}(\mathcal{I}) \big) \ : \ M \cap \text{sprt}(\mathcal{I}) \in \sigma_M(X)   \big\}
 \]
 This $H_M$-ultrafilter is often called the \textbf{measure derived from $\boldsymbol{\sigma_M}$ and $\boldsymbol{\text{sprt}(\mathcal{I})}$}.
\end{enumerate}
\end{fact}

If $\mathcal{I}$ is a normal fine ideal, there is a weaker notion than self-genericity isolated by Foreman-Magidor~\cite{MR1359154}:  a model $M \prec (H_\theta, \in, \mathcal{I})$ is \textbf{$\boldsymbol{\mathcal{I}}$-good} iff $M \cap \text{sprt}(\mathcal{I}) \in D$ for every $D \in M \cap \text{Dual}(\mathcal{I})$.  This is equivalent to requiring that the derived $\mathcal{U}(M,\mathcal{I})$ from above extends the dual of $\sigma_M^{-1}(\mathcal{I})$.  Let 
\[
\Omega(\mathcal{I}):= (2^{|\text{univ}(\mathcal{I})|})^+
\]
If $\mathcal{I}$ is normal and $\theta$ is sufficiently large, then the set $S^{\theta}_{\mathcal{I} \text{-good}}$ of $\mathcal{I}$-good structures is always stationary in $\wp(H_\theta)$ (Burke~\cite{MR1472122}).  The \textbf{conditional club filter relative to $\boldsymbol{\mathcal{I}}$} is defined to be the dual filter of
\[
\text{NS} \restriction S^{\Omega(\mathcal{I})}_{\mathcal{I}\text{-good}}
\]
Moreover, if $\mathcal{J}$ is the dual of the conditional club filter relative to $\mathcal{I}$, then $\mathcal{I}$ is a ``canonical projection of $\mathcal{J}$", which implies that the map $h_{\mathcal{I}, \mathcal{J}}$ defined by
\[
[S]_{\mathcal{I}} \mapsto  \big[ \{ M \in \text{univ}(\mathcal{J}) \ : \ M \cap \text{univ}(\mathcal{I}) \in S   \}  \big]_{\mathcal{J}} 
\]
is a well-defined boolean homomorphism from $\mathbb{B}_{\mathcal{I}} \to \mathbb{B}_{\mathcal{J}}$.

Notice that the collection of $\mathcal{I}$-good subsets of $H_{\Omega(\mathcal{I})}$ is a measure one set in the conditional club filter relative to $\mathcal{I}$.  Generalizing the terminology of Feng-Jech~\cite{MR1668171}, a collection $S$ of nonempty sets is called \textbf{$\boldsymbol{\mathcal{I}}$-projective stationary} iff for every $T \in \mathcal{I}^+$ and every algebra $\mathfrak{A}$ on $\cup S$, there is an $M \prec \mathfrak{A}$ such that $M \in S$ and $M \cap \text{sprt}(\mathcal{I}) \in T$.\footnote{The Feng-Jech notion of \emph{projective stationary} is the special case when $\mathcal{I} = \text{NS}_{\omega_1}$. }

\begin{definition}[Cox-Zeman~\cite{Cox_MALP}]
Let $\mathcal{I}$ be a normal ideal.  We say that \textbf{projective antichain catching holds for $\boldsymbol{\mathcal{I}}$}, abbreviated $\text{ProjCatch}(\mathcal{I})$, iff the set of $\mathcal{I}$-self-generic subsets of $H_{\Omega(\mathcal{I})}$ is $\mathcal{I}$-projective stationary.    We say that \textbf{club antichain catching holds for $\boldsymbol{\mathcal{I}}$}, abbreviated $\text{ClubCatch}(\mathcal{I})$, iff the set of $\mathcal{I}$-self-generic subsets of $H_{\Omega(\mathcal{I})}$ constitutes a measure one set in the conditional club filter relative to $\mathcal{I}$; equivalently,
\[
S^{\Omega(\mathcal{I})}_{\mathcal{I}\text{-self-generic}} \setminus S^{\Omega(\mathcal{I})}_{\mathcal{I}\text{-good}} \ \ \text{ is nonstationary}
\] 
\end{definition}

By standard arguments, the definition of projective antichain catching is unchanged if we use any large $\theta$ in place of $\Omega(\mathcal{I})$.  Another intermediate notion, called \emph{stationary antichain catching}, was defined in \cite{Cox_MALP} but will be of no use to us here.  The next fact is just a reformulation of Lemma 3.46 from Foreman~\cite{MattHandbook}.
\begin{fact}\label{fact_MattHBSatEquiv}
The following are equivalent for any normal fine ideal $\mathcal{I}$ such that $|\text{sprt}(\mathcal{I})| = |\text{univ}(\mathcal{I})|$:
\begin{itemize}
 \item $\mathcal{I}$ is saturated.
 \item $\text{ClubCatch}(\mathcal{I})$ holds.
\end{itemize}
\end{fact}

The key point of projective antichain catching is:

\begin{fact}[Cox-Zeman~\cite{Cox_MALP}, Lemmas 3.4 and 3.11]\label{fact_ProjCatch_TFAE}
The following are equivalent for any normal ideal $\mathcal{I}$:
\begin{enumerate}
 \item $\text{ProjectiveCatch}(\mathcal{I})$
 \item There exists a normal ideal $\mathcal{J}$ with support containing $H_{\Omega(\mathcal{I})}$, such that $\mathcal{J}$ canonically projects to $\mathcal{I}$ and the canonical boolean homomorphism from $\mathbb{B}_{\mathcal{I}}$ into $\mathbb{B}_{\mathcal{J}}$ is a regular embedding.
\end{enumerate}
\end{fact}
For the reader's convenience we sketch the proof of the forward direction of Fact \ref{fact_ProjCatch_TFAE}, which is the only direction that we will use:
\begin{proof}
(sketch):  Assume $\text{ProjectiveCatch}(\mathcal{I})$.  Let $\mathcal{J}$ be the restriction of the nonstationary ideal to the set $S^{\Omega(\mathcal{I})}_{\mathcal{I}\text{-self-generic}}$.\footnote{This $\mathcal{J}$ is the same as the restriction of (the dual of) the conditional club filter to $S^{\Omega(\mathcal{I})}_{\mathcal{I}\text{-self-generic}}$.  If $\mathcal{I}$ happens to be saturated and satisfy the other technical requirement listed in Fact \ref{fact_MattHBSatEquiv}, then the $\mathcal{J}$ defined here is in fact equal to (the dual of) the conditional club filter relative to $\mathcal{I}$.}  The assumption $\text{ProjectiveCatch}(\mathcal{I})$ ensures that the map $h$ defined by
\[
[S]_{\mathcal{I}} \mapsto  \big[ \{ M \in \text{univ}(\mathcal{J}) \ : \ M \cap \text{univ}(\mathcal{I}) \in S   \}  \big]_{\mathcal{J}} 
\]
is a boolean embedding from $\mathbb{B}_{\mathcal{I}}$ into $\mathbb{B}_{\mathcal{J}}$.  We need to prove that $h$ is a regular embedding, so fix a maximal antichain $A$ for $\mathbb{B}_{\mathcal{I}}$.  We view $A$ as a collection of $\mathcal{I}$-positive sets, rather than a collection of $\mathcal{I}$-equivalence classes, which should not cause any confusion (i.e.\ we view $h$ as mapping from the $\mathcal{I}$-positive sets into the $\mathcal{J}$-positive sets).  Assume $R \in \mathcal{J}^+$; we need to find some $S \in A$ such that $M \cap \text{sprt}(\mathcal{I})$ is an element of $S$ for $\mathcal{J}^+$-many $M \in R$.  We can without loss of generality assume that $A \in M$ for every $M \in R$.  Since the universe of $\mathcal{J}$ consists only of $\mathcal{I}$-self-generic structures, each $M \in R$ catches $A$; so there is an $S_M \in A \cap M$ such that $M \cap  \text{univ}(\mathcal{I}) \in S_M$.  By normality of $\mathcal{J}$, there is an $S \in A$ such that $S_M = S$ for $\mathcal{J}^+$ many $M \in R$.  Hence $S \in A$ and for $\mathcal{J}^+$-many $M \in R$, $M \cap \text{sprt}(\mathcal{I}) \in S$, which completes the proof.
\end{proof}

\begin{fact}[Corollary 3.12 of Cox-Zeman~\cite{Cox_MALP}]\label{fact_ImpliesNegForeman}
Suppose $\mathcal{I}$ is a normal fine ideal on some $Z \subseteq \wp(\kappa^{+n})$, and that Projective Antichain Catching holds for $\mathcal{I}$.  Then there exists a $Z'  \subseteq \wp(\kappa^{+(n+1)})$ and a normal ideal $\mathcal{J}$ on $Z'$ that canonically projects to $\mathcal{I}$, and such that the canonical homomorphism from $\wp(Z)/\mathcal{I} \to \wp(Z')/\mathcal{J}$ is a regular embedding.  In particular, $\mathcal{I}$ and $\mathcal{J}$ satisfy the assumptions of Question \ref{q_Foreman}.
\end{fact}
In order to show the reader how projective antichain catching is related to Question \ref{q_Foreman}, we sketch the proof of Fact \ref{fact_ImpliesNegForeman}:
\begin{proof}
(sketch)  Let $\mathcal{I}_0:= \mathcal{I}$, $Z_0:= Z$, and $H_0:= \kappa^{+n}$.  By Fact \ref{fact_ProjCatch_TFAE} there exists an ideal $\mathcal{I}_2$ with support $H_2 \supseteq H_{\Omega(\mathcal{I})}$ that projects to $\mathcal{I}_0$, and the induced homomorphism $e_{0,2}: \wp(Z_0)/\mathcal{I}_0 \to \wp(Z_2)/\mathcal{I}_2$ is a regular embedding, where $Z_2$ is the universe of $\mathcal{I}_2$ (so $Z_2 \subseteq \wp(H_2)$).   Since $H_2 \supset \kappa^{+(n+1)}$ then we can canonically project $\mathcal{I}_2$ to an ideal $\mathcal{I}_1$ with support $H_1:= \kappa^{+(n+1)}$ and universe $Z_1:= \{ M \cap H_1 \ : \ M \in Z_2 \}$.  It follows via routine verification that $\mathcal{I}_0$ is the projection of $\mathcal{I}_1$ to $H_0$, and that the following maps are well-defined Boolean homomorphisms for all $i < j \le 2$:
\[
e_{i,j}: \wp(Z_i)/\mathcal{I}_i \to \wp(Z_j)/\mathcal{I}_j\]
 defined by 
 \[[S]_{\mathcal{I}_i} \mapsto [\{ M \in Z_j \ : \ M \cap H_i \in S \}]_{\mathcal{I}_j}
 \]
Moreover, it is easy to check that
\begin{equation}\label{eq_Comp}
e_{0,2}= e_{1,2} \circ e_{0,1}
\end{equation}
Then \eqref{eq_Comp} and regularity of $e_{0,2}$ abstractly imply that $e_{0,1}$ is regular.  
\end{proof}

\subsection{Duality Theorem}\label{sec_Duality}

Suppose $W$ is a model of set theory, $\mathbb{Q}$ is a poset in $W$, and $\mathbb{Q}$ forces the existence of a $W$-normal ultrafilter $\dot{U}$ on, say, $\kappa$.  It follows that if, in $W$, one defines
\[
\mathcal{I}:= \{ A \subset \kappa \ : \  \Vdash_{\mathbb{Q}} \ A \notin \dot{U}  \}
\]
then $\mathcal{I}$ is a normal ideal, and the map
\[
e: [S]_{\mathcal{I}} \mapsto \llbracket \check{S} \in \dot{U} \rrbracket
\]
is a boolean homomorphism from $\wp(\kappa)/\mathcal{I} \to \text{RO}(\mathbb{Q})$.  In particular, the map $e$ is $\perp$-preserving, so if $\mathbb{Q}$ happens to be $\kappa^+$-cc, then so is $\wp(\kappa)/\mathcal{I}$; i.e.\ $\mathcal{I}$ is a saturated ideal in $W$.  This is how the original Kunen~\cite{MR495118} argument went.  

Producing presaturated ideals which are not saturated---which is one of the main tools for answering Questions \ref{q_Foreman} and \ref{q_Foreman_Potent}---is more delicate.  If the poset $\mathbb{Q}$ from above is merely $\kappa^+$-presaturated, then the fact that $e$ is $\perp$-preserving doesn't imply that $\wp(\kappa)/\mathcal{I}$ is also $\kappa^+$-presaturated.  For that, one typically needs to know that $e$ is a \textbf{regular} embedding.

For this reason, we extensively use the following theorem of Foreman.  If $\mathcal{J}$ is an ideal in $V$ and $V[G]$ is a forcing extension of $V$, we let $\bar{\mathcal{J}}$ denote the ideal generated by $\mathcal{J}$ in $V[G]$.  If $U$ is $(V,\mathbb{B}_{\mathcal{J}})$-generic then $j_U$ denotes the generic ultrapower of $V$ by $U$.  If $\mathbb{P}$ is a poset and $p \in \mathbb{P}$, then $\mathbb{P} \restriction p$ denotes the set of all conditions in $\mathbb{P}$ which are compatible with $p$; this is forcing equivalent to the cone below $p$.
\begin{theorem}[Duality Theorem, Foreman~\cite{MR3038554} Theorem 17 and Corollary 18]\label{thm_DualityThm}
Suppose $\mathcal{J}$ is a  normal precipitous ideal, $\dot{U}$ is the $\mathbb{B}_{\mathcal{J}}$-name for its generic object, $\mathbb{P}$ is a poset, and there is some $\mathbb{B}_{\mathcal{J}}$-name $\dot{m}$ for an element of $j_{\dot{U}}(\mathbb{P})$ such that the map $p \mapsto \big( 1, j_{\dot{U}}(\check{p}) \big)$ is a regular embedding from
\[
\mathbb{P} \to \mathbb{B}_{\mathcal{J}} \ * \ j_{\dot{U}}(\mathbb{P}) \restriction \dot{m}
\]

Suppose $[A] \in \mathbb{B}_{\mathcal{J}}$ decides a function $f$ representing $\dot{m}$; i.e.\ $f$ is a function and 
\[
[A] \Vdash_{\mathbb{B}_{\mathcal{J}}} \ \dot{m} = [\check{f}]_{\dot{U}}
\]
Let $\dot{S}$ be the $\mathbb{P}$-name for the set
\[
\{ z \in A \ : \ f(z) \in \dot{G}  \}
\]
where $\dot{G}$ is the $\mathbb{P}$-name for its generic object.

Then $\mathbb{P}$ forces that $\bar{\mathcal{J}}$ is normal and precipitous. Moreover there exists an isomorphism
\begin{equation*}
\phi: \mathrm{ro} \Big( \mathbb{P} \ * \ \mathbb{B}_{\bar{\mathcal{J}}} \restriction [\dot{S}]  \Big) \to \mathrm{ro}\Big( \mathbb{B}_{\mathcal{J}} \restriction [A] \ * \ j_{\dot{U}}(\mathbb{P}) \restriction \dot{m} \Big)
\end{equation*}
such that: whenever $G * \bar{U}$ is generic for $ \mathbb{P} \ * \ \mathbb{B}_{\bar{\mathcal{J}}} \restriction [\dot{S}]$ and $U*G'$ is the result of transferring $G*\bar{U}$ to a $\mathbb{B}_{\mathcal{J}} \restriction [A] \ * \ j_{\dot{U}}(\mathbb{P}) \restriction \dot{m}$-generic via $\phi$, then $U = \bar{U} \cap \wp^V(Z)$, $j_{\bar{U}}: V[G] \to \text{ult}(V[G], \bar{U})$ extends $j_U: V \to \text{ult}(V,U)$, and $j_U[G] \subseteq G'$.
\end{theorem}
  
It will be convenient to state the special case where $\mathbb{P}$ has nice chain condition.  In such cases, $\dot{m}$ from Theorem \ref{thm_DualityThm} is just the trivial condition, and there is no need to restrict below conditions:
\begin{fact}[\cite{MattHandbook}]\label{fact_DualityKappaCC}
Suppose $\mathcal{J}$ is a  normal precipitous ideal of uniform completeness $\kappa$, and $\mathbb{P}$ is a $\kappa$-cc poset. Then the map
\[
p \mapsto \big( 1, j_{\dot{U}}(p) \big)
\]
is a regular embedding from $\mathbb{P}$ into $\mathbb{B}_{\mathcal{J}} * j_{\dot{U}}(\mathbb{P})$.  In particular, the assumptions of Theorem \ref{thm_DualityThm} hold, and moreover all conclusions of Theorem \ref{thm_DualityThm} hold, where $A$ and $\dot{S}$ are the trivial conditions of their respective posets.
\end{fact}  
  

\section{Presaturation implies projective antichain catching}\label{sec_PresatAntichainCatch}
  
In this section we prove Theorem \ref{thm_Presat_implies_Catch}.

\begin{lemma}\label{lem_Substruct}
Suppose $\mathcal{I}$ is a normal ideal and $Y \prec (H_\theta, \in, \mathcal{I})$ is an $\mathcal{I}$-self-generic structure.  Define a map $F_Y$ on 
\[
Y \cap  \{ \mathcal{A} \ : \ \mathcal{A} \text{ is a maximal antichain in } \mathbb{B}_{\mathcal{I}} \}
\]
by letting $F_Y(\mathcal{A})$ be an $S \in \mathcal{A}$ such that $S \in Y$ and $Y \cap \text{sprt}(\mathcal{I}) \in S$; that is, $F_Y(\mathcal{A})$ is a set witnessing that $Y$ catches $\mathcal{A}$.\footnote{It is easy to see that there is a unique such witness, but that isn't important in the arguments to follow.}  Notice that $F_Y$ can be viewed as a predicate on the underlying set $Y$.

Suppose 
\[
X \prec (Y, \in, F_Y) \text{ and } \text{sprt}(\mathcal{I}) \cap Y \subset  X 
\]
Then $X$ is $\mathcal{I}$-self-generic.
\end{lemma}  
\begin{proof}
Let $\mathcal{A} \in X$ be a maximal antichain for $\wp(\kappa)/\mathcal{I}$.  Since $X \subset Y$ then $\mathcal{A} \in Y$, and thus $F_Y(\mathcal{A})$ is defined; let $S:=F_Y(\mathcal{A})$.  Then $S \in X$ because $X \prec (Y, \in, F_Y)$.  Since $\text{sprt}(\mathcal{I}) \cap Y \subset X$ and $X \subset Y$ then $z:=X \cap \text{sprt}(\mathcal{I}) = Y \cap \text{sprt}(\mathcal{I})$.  And $z \in S$ because $Y$ is $\mathcal{I}$-self-generic.
\end{proof}

\begin{lemma}\label{lem_GenericImageIsSG}
Suppose $\mathcal{I}$ is a normal ideal, $G$ is generic for $\mathbb{B}_{\mathcal{I}}$, and $j: V \to N$ is the generic ultrapower of $V$ by $G$.  Fix $\theta$ such that $\mathbb{B}_{\mathcal{I}} \in H_\theta$.  Then 
\[
V[G] \models \ j [H_\theta^V] \text{ is } j(\mathcal{I}) \text{-self-generic}
\]
\end{lemma}
\begin{proof}
Since $G$ is $(V,\mathbb{B}_{\mathcal{I}})$-generic, it is also $(H_\theta^V, \mathbb{B}_{\mathcal{I}})$-generic.  Notice also that 
\[
j \restriction H^V_\theta
\]
is the inverse of the Mostowski collapse of $j [H_\theta^V]$.  That $j[H_\theta^V]$ is $j(\mathcal{I})$-self-generic follows immediately by Fact \ref{fact_SelfGen_TFAE}.  
\end{proof} 
Note that in the proof of Lemma \ref{lem_GenericImageIsSG} we were working inside $V[G]$, not inside $N$; and although $j[H^V_\theta] \prec H^N_{j(\theta)}$, the set $j[H^V_\theta]$ might fail to be elementary in a hereditary initial segment of $V[G]$.  This is the reason that elementarity of the model was not assumed in the statement of Fact \ref{fact_SelfGen_TFAE}.

We now finish the proof of Theorem \ref{thm_Presat_implies_Catch}.  Suppose $\mathcal{I}$ is a normal presaturated ideal.  Fix a sufficiently large regular $\theta$, an algebra $\mathfrak{A}$ on $H_\theta$, and some $T \in \mathcal{I}^+$.  We need to find an $\mathcal{I}$-self-generic elementary substructure of $\mathfrak{A}$ whose intersection with $\text{sprt}(\mathcal{I})$ is in $T$.  Let $G$ be $\mathbb{B}_{\mathcal{I}}$-generic with $T \in G$, and let $j: V \to N$ be the generic ultrapower.  Since $T \in G$ and $\mathcal{I}$ is normal, then $j[\text{sprt}(\mathcal{I})]=[\text{id}]_G \in j(T)$. Then  
\begin{equation}
j[H_\theta^V] \cap j(\text{sprt}(\mathcal{I})) = j[\text{sprt}(\mathcal{I})] = [\text{id}]_G  \in j(T)
\end{equation}
Let $Y:= j[H_\theta^V]$.  By Lemma \ref{lem_GenericImageIsSG}, from the point of view of $V[G]$, $Y$ is a $j(\mathcal{I})$-self-generic set.  Elementarity of $j$, together with the fact that $\mathfrak{A}$ is in a countable language, guarantees that $Y \prec j(\mathfrak{A})$.   In particular, $j(\mathfrak{A}) \restriction Y$ is a well-defined structure on $Y$.

Work in $V[G]$.  Since $Y$ is $j(\mathcal{I})$-self-generic, we can define a map $F_Y$ on
\[
Y \cap \{ \mathcal{A} \ : \ \mathcal{A} \text{ is a maximal antichain for } \mathbb{B}_{j(\mathcal{I})} \}
\]
by letting $F_Y(\mathcal{A})$ be a set witnessing that $Y$ catches $\mathcal{A}$.\footnote{Again, there is only one such set, but this uniqueness isn't important to the argument.}  Define the following structure on $Y$: 
\begin{equation}
\mathfrak{B}:= \Big( j(\mathfrak{A}) \restriction Y\Big)^\frown F_Y
\end{equation}

Fix an $X \prec \mathfrak{B}$ such that $|X|=|[\text{id}]_G|$ and  $[\text{id}]_G \subset X$; then $X \cap [\text{id}]_G = Y \cap [\text{id}]_G$ and, by closure of $N$ under $|[\text{id}]_G|$ sequences from $V[G]$, we have that $X \in N$.  By Lemma \ref{lem_Substruct}, $X$ is $j(\mathcal{I})$-self-generic in $V[G]$, which is clearly downward absolute to $N$.  In summary, $X$ witnesses that
\[
N \models \ \exists X \prec j(\mathfrak{A}) \ \ X \cap j(\text{sprt}(\mathcal{I})) \in j(T) \text{ and } X \text{ is } j(\mathcal{I}) \text{-self-generic}
\]

By elementarity of $j$, 
\[
V \models \ \exists X \prec \mathfrak{A} \ \ X \cap \text{sprt}(\mathcal{I}) \in T \text{ and } X \text{ is } \mathcal{I} \text{-self-generic}
\]     
which completes the proof of Theorem \ref{thm_Presat_implies_Catch}.

\begin{corollary}\label{cor_PresatNonsatGivesAnswer}
Suppose $\mathcal{I}$ is a presaturated, non-saturated ideal with support $\kappa^{+n}$.\footnote{Recall by Definition \ref{def_SatPresat} this means that $\mathcal{I}$ is not $(\kappa^{+n})^+ = \kappa^{+(n+1)}$-saturated.}  Then $\mathcal{I}$ is a counterexample for Foreman's Question \ref{q_Foreman}. In other words, there is a normal $\mathcal{J}$ with support $\kappa^{+(n+1)}$ which canonically projects to $\mathcal{I}$ such that the corresponding boolean embedding from $\mathbb{B}_{\mathcal{I}} \to \mathbb{B}_{\mathcal{J}}$ is a regular embedding, yet $\mathcal{I}$ is not $\kappa^{+(n+1)}$-saturated. 
\end{corollary}
\begin{proof}
This follows immediately from Theorem \ref{thm_Presat_implies_Catch} and Fact \ref{fact_ImpliesNegForeman}.
\end{proof}

\section{Proof of Theorem \ref{thm_SolveHBquestion} and its corollaries}\label{sec_BT}
  
We prove the following generalization of several results from Baumgartner-Taylor~\cite{MR654852}.  Part \ref{item_PresPresatKillSat} is stated in a sufficiently general form to be used both for the proof of Theorem \ref{thm_SolveHBquestion}---which dealt with ``supercompact-type" ideals---and for the proof of Theorem \ref{thm_AnswerPotentAxioms}, which deals with ``huge-type" ideals.

\begin{theorem}\label{thm_KillSaturation}
Suppose $\mu$ is regular and $\mu^{<\mu} = \mu$.  Let $\kappa = \mu^+$.  Then there is a cardinal-preserving poset $\mathbb{P}$ such that:
\begin{enumerate}
 \item\label{item_NoSatIdealsKappa}   $\Vdash_{\mathbb{P}}$ ``There are no uniform, $\kappa$-complete saturated ideals on $\kappa$ (normal or otherwise)". 

 \item\label{item_PresPresatKillSat}  Suppose $\delta$ is a regular uncountable cardinal, $\mathcal{I}$ is a normal fine  $\delta$-presaturated ideal on some set $Z$, $2^{|Z|} < \delta^{+\omega}$, $\mathcal{I}$ has uniform completeness $\kappa$, 
 \[
 \Vdash_{\wp(Z)/\mathcal{I}} \ \text{cf}(\kappa) = \mu,
\] 
and 
\[
\Vdash_{\wp(Z)/\mathcal{I}} \ j_{\dot{U}}(\kappa) = \delta
\]
where $j_{\dot{U}}$ is the $\wp(Z)/\mathcal{I}$-name for the generic ultrapower embedding.    Then whenever $G$ is $(V,\mathbb{P})$-generic, letting $\bar{\mathcal{I}}$ denote the ideal generated by $\mathcal{I}$ in $V[G]$, there is some $\bar{\mathcal{I}}$-positive set $S$ such that:

 \begin{enumerate}
  \item $\bar{\mathcal{I}} \restriction S$ is not $\delta$-saturated; but
  \item $\bar{\mathcal{I}} \restriction S$ is $\delta$-presaturated. 
\end{enumerate}  
\end{enumerate}
\end{theorem}
The assumption that $\mathbb{B}_{\mathcal{I}}$ forces $j(\kappa) = \delta$ follows automatically in many standard kinds of $\delta$-presaturated ideals.  For example, suppose $\mathcal{I}$ is a supercompact-type ideal on some $\wp_\kappa(\lambda)$ where $\kappa$ is a successor cardinal.\footnote{Recall that by ``supercompact-type" we mean that the set $\{ x \in \wp_\kappa(\lambda) \ : \ x \cap \kappa \in \kappa \}$ has $\mathcal{I}$-measure one.}  Then, as mentioned in Section \ref{sec_Prelims}, $\lambda^+$-presaturation is the best possible, and in that case the generic ultrapower always sends $\kappa$ to $\lambda^+$.  Also notice that Theorem \ref{thm_KillSaturation}, in conjunction with Corollary \ref{cor_PresatNonsatGivesAnswer}, will imply Theorem \ref{thm_SolveHBquestion}.

The proof of Theorem \ref{thm_KillSaturation} generalizes some arguments of Baumgartner-Taylor~\cite{MR654852}, which in turn depended heavily on Baumgartner's poset for adding a club to $\omega_1$ with finite conditions.  This poset doesn't collapse any cardinals, and the generic club doesn't contain any infinite set from the ground model.  It is also, in modern terminology, a \emph{strongly proper} poset (see Section \ref{sec_Prelims}).   The following generalization is folklore, e.g.\ is mentioned in Abraham-Shelah~\cite{MR716625}.  We give a slightly different presentation that is more amenable to proving strong properness.  


\begin{definition}\label{def_TodorcPoset}
Let $\mu < \kappa$ be regular cardinals.  The poset $\mathbb{P}(\mu, \kappa)$ consists of pairs $(s,f)$ where:
\begin{itemize}
 \item $s \in [\kappa]^{<\mu}$.
 \item $f: s \to [\kappa]^{<\mu}$, and if $\xi<\xi'$ are both in $s$, then $f(\xi) \subseteq \xi'$.
\end{itemize} 

The ordering is defined by:  $(s',f') \le (s,f)$ iff $s' \supseteq s$ and $f'(\xi) \supseteq f(\xi)$ whenever $\xi \in s$.
\end{definition}

\begin{observation}\label{obs_SameDef}
If $V \subset W$ are transitive models of set theory and $V$ is closed under sequences of length $<\mu$ from $W$, then 
\[
\mathbb{P}^V(\mu,\kappa) = \mathbb{P}^W(\mu,\kappa)  
\]
\end{observation}

The following collects the facts we will need about $\mathbb{P}(\mu,\kappa)$.

\begin{lemma}\label{lem_FactsAboutP}
Let $\mu$ be a regular cardinals, let $\kappa = \mu^+$, and let $\mathbb{P} := \mathbb{P}(\mu,\kappa)$ from Definition \ref{def_TodorcPoset}.  Assume $\mu^{<\mu} = \mu$.  Then:
\begin{enumerate}
 \item\label{item_SizeKappa}  $|\mathbb{P}|=\kappa$; in particular $\mathbb{P}$ preserves $\kappa^+$.
 \item\label{item_DirectedClosure} $\mathbb{P}$ is ${<}\mu$-directed closed. (So all cardinals and cofinalities $\le \mu$ are preserved).
 \item\label{item_StrProper} If $\theta \ge \kappa^+$,  $M^* \prec (H_\theta, \in, \kappa)$, and $M^* \cap \kappa \in S^\kappa_\mu:= \kappa \cap \text{cof}(\mu)$, then $\mathbb{P}$ is strongly proper for $M^*$.    In particular, $\kappa$ is preserved.  In fact, for every $(s,f) \in M^*$, the object 
 \[
 \Big(s \cup \{ M^* \cap \kappa \}, f \cup \big\{ M^* \cap \kappa \mapsto \{M^* \cap \kappa\}  \big\} \Big)
 \]
  is a condition extending $(s,f)$, and is a strong master condition for $M^*$. 
 \item\label{item_NoContainClub} If $G \subset \mathbb{P}$ is generic over $V$, in $V[G]$ let
\[
C_G:= \{ \xi \ | \ \exists (s,f) \in G \ \ \xi \in s \}
\] 
Then:
\begin{enumerate} 
 \item $C_G$ is club in $\kappa$.
  \item If $X \in V$ and $|X|^V \ge \mu$, then $X \nsubseteq C_G$.
 \end{enumerate}
 \item\label{item_NotKappaCC} $\mathbb{P}$ is \textbf{not} $\kappa$-cc below any condition.
\end{enumerate}
\end{lemma}
\begin{proof}
If $\mu^{<\mu} = \mu$ then $\kappa^{<\mu} = \kappa$, which easily implies part \ref{item_SizeKappa}.  To see part \ref{item_DirectedClosure}, suppose $D$ is a directed subset of  $\mathbb{P}$ of size $<\mu$.  Let $t$ be the union of all first coordinates of $D$; note that $|t|<\mu$.   Define a function $h$ on $t$ as follows:  given a $\xi \in t$, let 
\[
D_\xi:= \{ (s,f) \in D  \ : \  \xi \in s \}
\]
and define 
\[
h(\xi):= \bigcup_{(s,f) \in D_\xi} f(\xi)
\]
We verify that $(t,h)$ is a condition; if so, then it is clearly below every member of $D$.  The assumption that $|D|<\mu$, and the definition of the poset, ensures that $|h(\xi)|<\mu$ for all $\xi$.  Now suppose $\xi < \xi'$ are both in $t$; we need to check that
\[
h(\xi) \subseteq \xi'
\]
and for that it suffices to prove that $f(\xi) \subseteq \xi'$ whenever $(s,f) \in D_\xi$.  So fix an $(s,f) \in D_\xi$, and fix some $(s', f') \in D$ witnessing that $\xi' \in t$; i.e.\ such that $\xi' \in s'$.  Since $D$ is directed there is some $(s^*, f^*)$ below both $(s,f)$ and $(s',f')$.  Then $\xi$ and $\xi'$ are both in $s^* = \text{dom}(f^*)$, $f(\xi) \subseteq f^*(\xi)$ because $(s^*, f^*) \le (s,f)$, and $f^*(\xi) \subseteq \xi'$ because $\xi$ and $\xi'$ are both in $s^*$ and $\xi < \xi'$.  It follows that $f(\xi) \subseteq \xi'$, completing the proof of part \ref{item_DirectedClosure}.

Next we prove part \ref{item_StrProper}.  Suppose $M^* \prec (H_\theta, \in, \kappa)$, $M^* \cap \kappa \in S^\kappa_\mu$, and that $(s,f)$ is a condition in $M^*$.  Since $\kappa^{<\mu} = \kappa$ and $M^* \prec (H_\theta, \in, \kappa, \mu)$, there is some $\phi \in M^*$ which is a bijection from $\kappa \to [\kappa]^{<\mu}$; since $M^* \cap \kappa \in \text{cof}(\mu)$ it follows easily that 
\begin{equation}\label{eq_M_closed}
{}^{<\mu} (M^* \cap  \kappa) \subset M^*
\end{equation}
Since $|s| < \mu \subset M^* \cap \kappa$ then $s \subset M^*$ and thus $M^* \cap \kappa \notin s = \text{dom}(f)$.   Also, if $\xi \in s$ then $f(\xi) \in M^* \cap [\kappa]^{<\mu}$, and again the elementarity of $M^*$ and the fact that $\mu \subset M^*$ imply that $f(\xi) \subset M^* \cap \kappa$.  It follows that 
 \begin{equation}\label{eq_SMC}
 \Big(s \cup \{ M^* \cap \kappa \}, f \cup \big\{ (M^* \cap \kappa,  \{M^* \cap \kappa\})  \big\} \Big)
 \end{equation}
is a condition extending $(s,f)$.  If $(t,h)$ extends the condition \eqref{eq_SMC}, then $t_{M^*}:= t \cap M^*$ is a ${<}\mu$-sized subset of $M^* \cap \kappa$, and is thus an element of $M^*$ by \eqref{eq_M_closed}.  Also since $M^* \cap \kappa \in t$, conditionhood of $(t,h)$ implies that $h \restriction t_{M^*}$ maps into $[M^* \cap \kappa]^{<\mu}$.   It follows by \eqref{eq_M_closed} that $(t_{M^*}, h \restriction t_{M^*})$ is an element of $M^* \cap \mathbb{P}$.   To show that it is a reduct of $(t,h)$ into $M^* \cap \mathbb{P}$, assume $(u,g)$ is a condition in $M^* \cap \mathbb{P}$ extending  $(t_{M^*}, h \restriction t_{M^*})$.  Define a function $F$ on $u \cup t$ by $F(\xi) = g(\xi)$ if $\xi \in u$, and $F(\xi) = h(\xi)$ otherwise.  It is routine to check that $(u \cup t, F)$ is a condition and extends both $(u,g)$ and $(t,h)$.  This completes part \ref{item_StrProper}.

For part \ref{item_NoContainClub}, that $C_G$ is cofinal in $\kappa$ follows from an easy genericity argument, since if $(s,f)$ is a condition and $\beta$ is sufficiently large relative to the range of $f$ then $\big(s \cup \{ \beta \}, f \cup \{ (\beta,\emptyset) \} \big)$ is a condition.  To see that $C_G$ is closed,
let $(s,f)$ be a condition, $\alpha <\kappa$, and suppose that 
\begin{equation}\label{eq_AlphaNotThere}
(s,f) \Vdash \check{\alpha} \notin \dot{C}_{\dot{G}}, 
\end{equation}
which implies that $\alpha \notin s$.  We argue that $(s,f) \Vdash ``\check{\alpha}$ is not a limit of $\dot{C}_{\dot{G}}$".    Suppose first that for all $\xi \in s \cap \alpha$, $f(\xi) \subset \alpha$.  Then $\big(s \cup \{ \alpha \}, f \cup \{ (\alpha,\emptyset) \} \big)$ is a condition, contradicting \eqref{eq_AlphaNotThere}.  Thus there is some $\xi \in s \cap \alpha$ such that $f(\xi) \nsubseteq \alpha$. $\xi$ must be the largest element of $s \cap \alpha$, and $(s,f)$ forces that $\xi$ is the largest element of $C_G \cap \alpha$.

%
%

Finally, let $(s,f)$ be a condition, and let $X \in V$ be such that $|X| \ge \mu$.  We find an extension of $(s,f)$ forcing that $\check{X} \nsubseteq \dot{C}_{\dot{G}}$.  
By taking an initial segment of $X$ we may without loss of generality assume $\text{otp}(X) = \mu$, so in particular $\text{cf}(\text{sup}(X)) = \mu$.  Since $|s|<\mu$ and $\text{sup}(X) \in \text{cof}(\mu)$, $s \cap \text{sup}(X)$ is bounded below $\text{sup}(X)$.  If there is some $\xi \in s \cap \text{sup}(X)$ such that $f(\xi) \nsubseteq \text{sup}(X)$ then $(s,f)$ already forces that $\dot{C}_{\dot{G}} \cap \check{X}$ is bounded below $\text{sup}(X)$; otherwise we can find a $\zeta_X < \text{sup}(X)$ and define a condition $(s',f')$ extending $(s,f)$ which forces that $\dot{C}_{\dot{G}} \cap \text{sup}(X)$ has maximum element $\zeta_X$.

To show part \ref{item_NotKappaCC}, note that for any $\beta < \kappa$,
\[
\Big\{ \big( \{ \beta, \alpha+1 \}, \{ (\beta,\{\alpha\}), (\alpha+1,\emptyset) \}  \big): \beta < \alpha < \kappa \Big\}
\]
is an antichain.  If $(s,f)$ is any condition, then all members of the above antichain are compatible with $(s,f)$ if $\beta$ is sufficiently large.
\end{proof}


We split the proof of Theorem \ref{thm_KillSaturation} into two parts, since they are completely independent of each other.

\subsection{Part \ref{item_NoSatIdealsKappa} of Theorem \ref{thm_KillSaturation}}

In this subsection we prove part \ref{item_NoSatIdealsKappa} of Theorem \ref{thm_KillSaturation}.  We roughly follow the outline of the proof of Theorem 3.5 of Baumgartner-Taylor~\cite{MR654852}.  Assume $\mu^{<\mu} = \mu$ is regular and let $\kappa = \mu^+$.  Let $\mathbb{P}:= \mathbb{P}(\mu,\kappa)$ be as in Definition \ref{def_TodorcPoset}.  Let $G$ be generic for $\mathbb{P}$, and suppose for a contradiction that in $V[G]$ there exists a $\kappa$-complete, $\kappa^+$-saturated, uniform ideal $\mathcal{J}$ on $\kappa$.   Let $U$ be generic for $\wp(\kappa)/\mathcal{J}$ over $V[G]$ and $j:V[G] \to \text{ult}(V[G],U)$ be the generic ultrapower by $U$.   Let $N:= \bigcup_{\alpha \in \text{ORD}} j(V_\alpha)$; then $j(\mathbb{P}) \in N$ and $\text{ult}(V[G], U)$ is of the form $N[g']$ for some $g' \in V[G*U]$ which is $\big(N, j(\mathbb{P})\big)$-generic.   Then $\tau:= j(\kappa) = j(\mu^{+V[G]}) = \mu^{+N[g']}$, and since $\wp(\kappa)/\mathcal{J}$ is $\kappa^+$-cc if follows that $\tau =  \kappa^{+V[G]} = \kappa^{+V}$, where the last equality is by part \ref{item_SizeKappa} of Lemma \ref{lem_FactsAboutP}.   By Fact \ref{fact_SatNonNormal}, $N[g']$ is closed under ${<}\tau$ sequences in $V[G*U]$.   

By part \ref{item_NoContainClub} of Lemma \ref{lem_FactsAboutP}, $C':= C_{g'}$ has the following properties:

\begin{equation}\label{eq_NgprimeBelieves}
\begin{split}
N[g'] \models & C' \text{ is club in } \tau, \text{ and }  \forall X \in N \ \  |X|^N \ge \mu \implies X \nsubseteq C'
\end{split}
\end{equation}
Since $V[G*U]$ is a $\tau$-cc forcing extension of $V$, there is a $D \in V$ such that 
\begin{equation}\label{eq_D_contained}
D \text{ is club in } \tau \text{ and } D \subseteq C' 
\end{equation}
Let $E$ be any subset of $D$
in $V$ such that $|E|^V =\kappa$, and let $\alpha:=\text{sup}(E)$; notice that $\alpha < \tau$.   Fix a $\phi \in V$ which is a bijection from $\kappa$ to $\alpha$.  Then $j(\phi) \in N$ by elementarity of $j : V[G] \to N[g']$, and it follows that $j [\alpha]$ is the range of $j(\phi) \restriction \kappa$ and thus $j \restriction \alpha \in N$.   Also $j(E) \in N$ since $E \in V$.  Since $E = \{ \beta < \alpha : j(\beta) \in j(E) \}$, $E$ is an element of $N$ as well.  But $|E|^V = \kappa$ implies $|E|^N \ge \mu$; since $E \subset C'$ this contradicts \eqref{eq_NgprimeBelieves} and completes the proof of part \ref{item_NoSatIdealsKappa} of Theorem \ref{thm_KillSaturation}.

It's natural to wonder if $\mathbb{P}(\mu,\kappa)$ also forces that there are no saturated ideals on, say, $\wp_\kappa(\lambda)$, where $\lambda \ge \kappa^+$.   
The end of the proof above relied on being able to conclude that $j[\text{sup}(E)]$ is an element of $N$.   In this more general setting, $\text{sup}(E)$ could be larger than $\kappa^+$, and it would not be clear how to obtain that $j[\text{sup}(E)] \in N$.  If the ideal is normal then $j[\text{sup}(E)]$ will be an element of $N[g']$, but that does not seem to be enough for the proof to go through. 


\subsection{Part \ref{item_PresPresatKillSat} of Theorem \ref{thm_KillSaturation}}\label{sec_PreservePresat}

In this section we prove part \ref{item_PresPresatKillSat} of Theorem \ref{thm_KillSaturation}.  Fix a $\delta$, $\kappa$, and $\mathcal{I}$ as in the assumptions of part \ref{item_PresPresatKillSat}.  Recall the background assumption of Theorem \ref{thm_KillSaturation} that $\kappa = \mu^+$ and $\mu^{<\mu} = \mu$.  Let $\mathbb{P}:= \mathbb{P}(\mu,\kappa)$ as in Definition \ref{def_TodorcPoset}.

Let $U$ be generic for $\mathbb{B}_{\mathcal{I}}$ over $V$ and $j: V \to N$ the generic ultrapower embedding.  By the assumptions of part \ref{item_PresPresatKillSat} of Theorem \ref{thm_KillSaturation}, 
\begin{equation}\label{eq_CofMu}
\text{cf}^{V[U]}(\kappa) = \mu
\end{equation}

Note that $\mathbb{P} \in H_{\kappa^+}$ (by part \ref{item_SizeKappa} of Lemma \ref{lem_FactsAboutP}), so $\mathbb{P}$-genericity over $V$ is equivalent to $\mathbb{P}$-genericity over $H^V_{\kappa^+}$.  In $V[U]$ consider the object $j[H^V_{\kappa^+}]$.  Notice that
\begin{equation}\label{eq_PointwiseImage}
j[H^V_{\kappa^+}] \cap j(\mathbb{P}) = j[\mathbb{P}]
\end{equation}

If the support of $\mathcal{I}$ is sufficiently larger than $\kappa$ then $j[H^V_{\kappa^+}] \in N$, which would somewhat simplify the proof below.\footnote{Namely, if $j[H^V_{\kappa^+}] \in N$ then one would be able to work directly in $N$ to find an $\big( j[H^V_{\kappa^+}], j(\mathbb{P}) \big)$-strong master condition $m$ in $N$, and then simply cite Lemma \ref{lem_DiffModelsSameInt} to conclude that $m$ is a $\big( j[H^V_{\kappa^+}], j(\mathbb{P}) \big)$-strong master condition from the point of view of $V[U]$ as well.  In that scenario, there would be no need to use the intermediary set $H$ defined in the next paragraph.}  However, if the support of $\mathcal{I}$ is, say, $\kappa$, then $j[H^V_{\kappa^+}] \notin N$.  We give a uniform treatment which works whether or not $j[H^V_{\kappa^+}]$ is an element of $N$.  

Back in $V$ fix some $H \prec (H_{\kappa^+}, \in, \mathbb{P})$ such that $|H|=\kappa \subset H$; since $|\mathbb{P}|=\kappa$ then $\mathbb{P} \subset H$ by elementarity of $H$ and the fact that $\kappa \subset H$.  Since $|H|^V = \kappa$, $j[H] \in N$, and since $\mathbb{P} \subset H$,
$j[H] \cap j(\mathbb{P}) = j[\mathbb{P}]$.
Together with \eqref{eq_PointwiseImage} in particular we have
\begin{equation}\label{eq_SameIntersection}
j[H^V_{\kappa^+}] \cap j(\mathbb{P}) = j[H] \cap j(\mathbb{P})
\end{equation}

Now $j[H] \cap j(\kappa) = \kappa$, and \eqref{eq_CofMu} implies that $\text{cf}^N(\kappa) = \mu$.  Then by part \ref{item_StrProper} of Lemma \ref{lem_FactsAboutP} applied inside $N$, there is an $m \in j(\mathbb{P})$ such that 
\begin{equation}
N \models  \ \ m \text{ is an } \big(j[H], j(\mathbb{P}) \big) \text{-strong master condition}
\end{equation} 
Since $N$ is a transitive submodel of $V[U]$, Corollary \ref{cor_Absolute} implies 
\begin{equation*}
V[U] \models  \ \ m \text{ is an } \big(j[H], j(\mathbb{P}) \big) \text{-strong master condition}
\end{equation*} 
Then \eqref{eq_SameIntersection} and Lemma \ref{lem_DiffModelsSameInt} together imply
\begin{equation}
V[U] \models  \ \ m \text{ is a } \big(j[H^V_{\kappa^+}], j(\mathbb{P}) \big) \text{-strong master condition}
\end{equation}
So if $G'$ is generic over $V[U]$ for $j(\mathbb{P})$ and $m \in G'$, then $G:= j^{-1}[G']$ is generic over $H^V_{\kappa^+}$, and thus over $V$, for the poset $\mathbb{P}$.  Hence, letting $\dot{m}$ be a $\mathbb{B}_{\mathcal{I}}$-name for $m$, the map $p \mapsto (1, j_{\dot{U}}(p))$ is a regular embedding from
\[
\mathbb{P} \to \mathbb{B}_{\mathcal{I}} \ * \ j_{\dot{U}}(\mathbb{P}) \restriction \dot{m}
\] 
It follows by the Duality Theorem (Theorem \ref{thm_DualityThm}) that in $V$ there is $A \in \mathcal I^+$ and in $V[G]$ there is an $S \in \bar{\mathcal{I}}^+$ such that $\bar{\mathcal{I}} \restriction S$ is precipitous, and the following forcing equivalence holds:
\begin{equation}\label{eq_ForcingEquiv}
\mathbb P * \mathbb{B}_{\bar{\mathcal{I}}} \restriction [\dot S] \ \sim \ \mathbb{B}_{\mathcal{I}} \restriction [A]  \ * \ j_{\dot{U}}(\mathbb{P}) \restriction \dot{m} 
\end{equation}


By part \ref{item_NotKappaCC} of Lemma \ref{lem_FactsAboutP} applied inside $N$, the poset $j_{\dot{U}}(\mathbb{P}) \restriction m$ is not $\delta= j(\kappa)$-cc,\footnote{Recall that one of the assumptions of part \ref{item_PresPresatKillSat} of the theorem is that $j(\kappa)$ is forced to be equal to $\delta$. }  which is upwards absolute to $V[U]$, and it follows by \eqref{eq_ForcingEquiv} that $\bar{\mathcal{I}} \restriction S$ is not $\delta$-saturated in $V[G]$.

Finally we prove that, in $V[G]$, the poset $\mathbb{B}_{\bar{\mathcal{I}} \restriction S}$ is $\delta$-presaturated as in Definition \ref{def_PresatPoset}.  By \eqref{eq_ForcingEquiv} this is equivalent to showing that $\mathbb{B}_{\mathcal{I}} \restriction [A]  \ * \ j_{\dot{U}}(\mathbb{P}) \restriction \dot{m} $ is $\delta$-presaturated in $V[G]$, and by Lemma \ref{lem_MonroeQuotient} it in turn suffices to prove that 
\begin{equation}\label{eq_PresatInV}
V \models \ \mathbb{B}_{\mathcal{I}} \ * \ j_{\dot{U}}(\mathbb{P}) \text{ is } \delta \text{-presaturated}
\end{equation}

In some situations where the structure of $\mathbb{B}_{\mathcal{I}}$ is better-known, \eqref{eq_PresatInV} can be proved directly.  For example, if we had assumed that, in $V$, $\mathbb{B}_{\mathcal{I}}$ is $\delta$-proper on a stationary set of internally approachable structures, then one could use arguments from Section \ref{sec_Prelims} to  to prove \eqref{eq_PresatInV} directly.  We also do not know if $\delta$-presaturation is closed under 2-step iterations in general (see Question \ref{q_2stepPresat}); if so then again one could prove \eqref{eq_PresatInV} directly. 

However the hypotheses of Theorem \ref{thm_KillSaturation} do not specify the structure of $\mathbb{B}_{\mathcal{I}}$ nor any forcing-theoretic properties beyond $\delta$-presaturation, so in the absence of a positive solution to Question \ref{q_2stepPresat} we must rely instead on some other facts from Section \ref{sec_Prelims} in order to prove \eqref{eq_PresatInV}.

\begin{nonGlobalClaim}\label{clm_HoldsInVG}
\[
V[G] \models \ \forall n \in \omega \ \  \Vdash_{\mathbb{B}_{\bar{\mathcal{I}}} \restriction S}  \ \mathrm{cf}\big( \delta^{+n} \big)^{V[G]} \ge \delta
\]
\end{nonGlobalClaim}
\begin{proof}
(of Claim \ref{clm_HoldsInVG})   By \eqref{eq_ForcingEquiv} and the fact that $\mathbb{P}$ preserves all cardinals, it suffices to prove that, over $V$, the poset $\mathbb{B}_{\mathcal{I}}  \ * \ j_{\dot{U}}(\mathbb{P}) \restriction \dot{m}$ forces that $\text{cf}\big(  \delta^{+n})^V \ge \delta$ for all $n \in \omega$.   The first step $\mathbb{B}_{\mathcal{I}}$  is $\delta$-presaturated by assumption, so by Fact \ref{fact_DeltaPresatPreserveCof} it forces $\text{cf}\big( \delta^{+n} \big)^V \ge \delta$ for all $n \in \omega$.  Now consider the second step.  If $U$ is $(V,\mathbb{B}_{\mathcal{I}})$-generic and $j: V \to_U N$ is the generic ultrapower, then Corollary \ref{cor_KappaSuccessorCard} and Observation \ref{obs_SameDef} ensure that 
\begin{equation}\label{eq_jP_is_presat}
j(\mathbb{P}) = \mathbb{P}^N(\mu,j(\kappa)) = \mathbb{P}^{V[U]}(\mu,j(\kappa))
\end{equation}
So by part \ref{item_StrProper} of Lemma \ref{lem_FactsAboutP} applied inside $V[U]$, $j(\mathbb{P})$ is (strongly) proper with respect to stationarily many models in $\wp^*_{\delta}\big( H^{V[U]}_\theta \big)$.   Thus $j(\mathbb{P})$ is $\delta$-presaturated in $V[U]$ by Fact \ref{fact_ProperImpliesPresat}.   Finally, by Fact \ref{fact_DeltaPresatPreserveCof} it follows that
\[
V[U] \models \ \Vdash_{j(\mathbb{P})} \ \forall n \in \omega \ \  \text{cf}\big( \delta^{+n} \big)^{V[U]} \ge \delta
\]
\end{proof}

The assumption that $2^{|Z|} < \delta^{+\omega}$ held in the ground model is clearly preserved in the extension $V[G]$ because $|\mathbb{P}| \le \delta$ and preserves all cardinals.   So
\[
V[G] \models  \  \ |\mathbb{B}_{\bar{\mathcal{I}}}| < \delta^{+\omega}
\]

These facts, together with Claim \ref{clm_HoldsInVG} and Fact \ref{fact_DeltaPlusOmega}, imply that  $\bar{\mathcal{I}} \restriction S$ is $\delta$-presaturated in $V[G]$.

\section{Proof of Theorem \ref{thm_Singular}}\label{sec_Singular}

Unlike the case with successors of regular cardinals, if $\kappa$ is the successor of a singular cardinal, we do not have a way to turn an arbitrary saturated ideal of completeness $\kappa$ into a presaturated, nonsaturated ideal (even in presence of GCH).  That is, we do not have an analogue of Theorem \ref{thm_KillSaturation} in this context.  However, we can construct such an ideal directly, which will prove Theorem \ref{thm_Singular} and thus provide a counterexample for Question \ref{q_Foreman} in the ``successor of a singular" context.

We want, for any fixed $n \in \omega$, to construct a model with an ideal on $\wp_{\aleph_{\omega+1}}(\aleph_{\omega + 1}^{ + n})$ which is $\aleph_{\omega+1}^{+(n+1)}$-presaturated, but not $\aleph_{\omega+1}^{+(n+1)}$-saturated.  Foreman~\cite{MR819932} produced a model with an $\aleph_{\omega+1}^+$-saturated ideal on $\aleph_{\omega + 1}$, and the ``skipping cardinals" technique of Magidor~\cite{MR526312} (also described in Section 7.9 of \cite{MattHandbook}) allows one to extend Foreman's method to produce an $\aleph_{\omega+1}^{+(n+1)}$-saturated ideal on $\wp_{\aleph_{\omega+1}}(\aleph_{\omega + 1}^{ +n })$ for any fixed $n \in \omega$.   However, it's not clear if one can start from that model and then force to kill the $\aleph_{\omega+1}^{+(n+1)}$-saturation, while preserving the $\aleph_{\omega+1}^{+(n+1)}$-presaturation.  

To get around this issue, we consider a different ideal in that very same model, which \emph{already} is $\aleph_{\omega+1}^{+(n+1)}$-presaturated and not $\aleph_{\omega+1}^{+(n+1)}$-saturated.    This introduces some complications because while $\delta$-saturation is preserved downward by any $\perp$-preserving embedding---e.g.\ any boolean embedding arising from an ideal projection----$\delta$-presaturation is not.  To preserve $\delta$-presaturation downward, one must generally arrange that the embeddings are regular embeddings, rather than merely being $\perp$-preserving.

Assume GCH, $\mu$ is indestructibly supercompact, and $\kappa > \mu$ is an almost huge cardinal, and let $j:V \to N$ be an almost huge \emph{tower embedding} witnessing the almost hugeness of $\kappa$.  That is, $j$ is an almost huge embedding and, letting  $\delta = j(\kappa)$, every element of $N$ is of the form
\begin{equation}\label{eq_TowerForm}
j(F)(j [\gamma])
\end{equation}
where $\gamma < \delta$ and $F: \wp_\kappa(\gamma) \to V$.\footnote{See \cite{MR526312} for details on almost huge tower embeddings.} By the Kunen/Laver Theorem \ref{thm_KunenLaver}, there is a $\mu$-directed closed, $\kappa$-cc poset $\mathbb{P} \subset V_\kappa$ and a regular embedding  
\[
e: \mathbb{P} * \dot{\mathbb{S}} \to j(\mathbb{P}) \text{ is a regular embedding}
\]   
where $e$ is the identity on $\mathbb{P}$ and  $\dot{\mathbb{S}}$ is the $\mathbb{P}$-name for the $\kappa^{+n}$-closed Silver collapse to turn $\delta$ into $\kappa^{+(n+1)}$.

Let $G*H$ be $\big(V,\mathbb{P}*\dot{\mathbb{S}}\big)$-generic.  Suppose $G'$ is generic over $V[G*H]$ for $j(\mathbb{P})/G*H$.  In $V[G']$ the map $j$ can be extended to
\[
j_{G'}: V[G] \to N[G']
\]
due to the $\kappa$-cc of $\mathbb{P}$.  Let $j_{G'}(\mathbb{S})/j_{G'} " H$ denote the set of conditions in $j_{G'}(\mathbb{S})$ which are compatible with every member of $j_{G'}"H$, with ordering inherited from $j_{G'}(\mathbb{S})$.  By Magidor's Theorem \ref{thm_MagidorVariation}, if $H'$ is generic over $V[G']$ for $j_{G'}(\mathbb{S})/j_{G'} " H$, then $H'$ is generic over $N[G']$ for $j_{G'}(\mathbb{S})$, and in $V[G'*H']$ the map $j_{G'}$ can be further extended to
\[
j_{G'*H'}: V[G*H] \to N[G'*H']
\]
Recall that $j$ was a tower embedding in the ground model.  It follows easily that $j_{G'*H'}$ is also a tower embedding by a tower of height $\delta$; but since $\delta$ is a successor cardinal in $V[G*H]$---namely $\delta = \kappa^{+(n+1)}$---then actually we have that $j_{G'*H'}$ is an ultrapower embedding by a normal measure:
\begin{nonGlobalClaim}\label{clm_UltMap}
The map $j_{G'*H'}$ is an ultrapower map by a $V[G*H]$-normal measure on $\Big( \wp_\kappa(\kappa^{+n}) \Big)^{V[G*H]}$.  That is, every element of $N[G'*H']$ is of the form
\[
j_{G'*H'}(f) \big( j[\kappa^{+n}] \big)
\]
for some $f \in V[G*H]$ with $f: P_\kappa(\kappa^{+n}) \to V[G*H]$. 
\end{nonGlobalClaim}
\begin{proof}
(of Claim \ref{clm_UltMap}):  Since everything in $N$ has the form \eqref{eq_TowerForm}, then an arbitrary element $y$ of $N[G'*H']$ has the form $\big(j(F)(j[\gamma])\big)_{G'*H'}$ for some $\gamma  < \delta$ and $F \in V \cap \wp_\kappa(\gamma)$.  In $V[G*H]$ define $F^*$ on $\wp_\kappa(\gamma)$ by letting $F^*(z) = F(z)_{G*H}$ if $F(z)$ is a $\mathbb{P}*\dot{\mathbb{S}}$-name, and undefined otherwise.  Note that the form for $y$ above, together with elementarity of $j_{G'*H'}$, imply in particular that
\begin{equation}
j_{G'*H'}(F^*)\big(  j[\gamma]\big) = y 
\end{equation}

 Since $\delta = \kappa^{+(n+1)}$ in $V[G*H]$, then in $V[G*H]$ there is some surjection $\phi: \kappa^{+n} \to \gamma$.   An easy calculation using surjectivity of $\phi$ and elementarity of $j_{G'*H'}$ shows
 \begin{equation}\label{eq_ImageComputation}
 j[\gamma] = j_{G'*H'}[\gamma] = j_{G'*H'}(\phi) \big[ j[\kappa^{+n}] \big]
 \end{equation}
 where the first equality is because $j_{G'*H'}$ extends $j$.
 
 In $V[G*H]$ define a function $f$ on $\wp_\kappa(\kappa^{+n})$ by
\[
f(u) := F^*(\phi[u]) 
\]
Then
\begin{align*}
\begin{split}
j_{G'*H'}(f) \big( j[\kappa^{+n}] \big) = j_{G'*H'}(F^*) \big(  j_{G'*H'}(\phi)\big[ j[\kappa^{+n}] \big] \big) \\
=j_{G'*H'}(F^*)\big( j[\gamma]  \big) =y
\end{split}
\end{align*}
where the next-to-last equality is by \eqref{eq_ImageComputation}.
\end{proof}

Thus in $V[G*H]$ the name for the generic embedding $j_{\dot{G}'*\dot{H}'}$ induces a normal ideal $\mathcal{I}$ on $\wp_\kappa(\kappa^{+n})$.  By a minor modification of the proof of the Duality Theorem (Theorem \ref{thm_DualityThm}), the following forcing equivalence holds in $V[G*H]$:
\begin{equation}\label{eq_Equiv_Singular}
\wp\big( \wp_\kappa(\kappa^{+n}) \big)/\mathcal{I} \ \sim \ \frac{j(\mathbb{P})}{G*H} \ * \ j_{\dot{G}'}(\mathbb{S}) / j_{\dot{G}'} [H]
\end{equation}

The poset $j(\mathbb{P})/G*H$ is $\delta = \kappa^{+(n+1)}$-cc in $V[G*H]$, and forces that $j_{\dot{G}'}(\mathbb{S}) / j_{\dot{G}'}[ H]$ is ${<}\delta$-closed and \textbf{not} $\delta$-cc.  So by \eqref{eq_Equiv_Singular} and Fact \ref{fact_ProperImpliesPresat} we have:
\begin{equation}\label{eq_InVGH_presat}
V[G*H] \models \ \mathbb{B}_{\mathcal{I}} \text{ is } \delta \text{-proper on a stationary set, but not } \delta \text{-saturated}
\end{equation}

Since $\mathbb{P}*\dot{\mathbb{S}}$ is ${<}\mu$ directed closed, and $\mu$ was indestructibly supercompact in the ground model, then $\mu$ is still supercompact in $V[G*H]$.  Set $W:= V[G*H]$, and work in $W$ for the remainder of the proof.  Let $\mathbb{R}$ be the Prikry forcing with a guiding generic to turn $\mu$ into $\aleph_{\omega}$ (see Section 2 of  \cite{MR3681102}).  This forcing adds a cofinal $\omega$-sequence to $\mu$, while collapsing cardinals in between the members of the sequence, and preserving that $\mu$ is a cardinal.  Conditions in the forcing take the form $(s,A)$, where $s \in V_\mu$, and any collection of ${<}\mu$ many conditions with the same first coordinate have a common extension with the same first coordinate.  Therefore $\mathbb{R}$ can be expressed in the form
\[
\mathbb{R} = \bigcup_{\xi < \mu} F_\xi
\]
where each $F_\xi$ is a filter.  It follows that if $U$ is generic over $W$ for $\mathbb{B}_{\mathcal{I}}$ and $j: W \to W'$ is the generic ultrapower, then $j(\langle F_\xi \ : \ \xi < \mu \rangle)$ is a sequence of $j(\mu) = \mu$ many filters whose union is $j(\mathbb{R})$.  This is upward absolute to $W[U]$, and in particular 
\begin{equation}\label{eq_jR_delta_cc}
W[U] \models \ j(\mathbb{R}) \text{ is } \delta = j(\kappa) \text{-cc}
\end{equation}
It follows by Fact \ref{fact_DualityKappaCC}---applied from the point of view of the model $W$---that $\bar{\mathcal{I}}$ is forced by $\mathbb{R}$ to be precipitous, and that the following forcing equivalence holds in $W$:
\begin{equation}\label{eq_MoreEquivSingular}
\mathbb{R} \ * \ \mathbb{B}_{\bar{\mathcal{I}}} \ \ \sim \ \  \mathbb{B}_{\mathcal{I}} \ * \ j_{\dot{U}}(\mathbb{R})  =: \mathbb{S}
\end{equation}

In $W$, consider the 2-step poset $\mathbb{S}$ on the right side of \eqref{eq_MoreEquivSingular}.  By \eqref{eq_InVGH_presat}, the first step is $\delta$-proper on a stationary set, and by \eqref{eq_jR_delta_cc} the second step is forced to be $\delta$-cc.  By Fact \ref{fact_IterationsDeltaProper}, $\mathbb{S}$  is $\delta$-proper on a stationary set and in particular $\delta$-presaturated (by Fact \ref{fact_ProperImpliesPresat}).  It follows by \eqref{eq_MoreEquivSingular} and Lemma \ref{lem_MonroeQuotient} that if $K$ is $(W,\mathbb{R})$-generic, then
\[
W[K] \models   \ \mathbb{B}_{\bar{\mathcal{I}}} \text{ is } \delta\text{-presaturated}
\]

Finally, the equivalence above yields that 
\begin{equation}\label{eq_YetAnotherEquiv}
W[K] \models  \ \mathbb{B}_{\bar{\mathcal{I}}} \ \sim \ \Big( \mathbb{B}_{\mathcal{I}} \ * \ j_{\dot{U}}(\mathbb{R}) \Big) / G_{\mathbb{R}}
\end{equation}

Recall from \eqref{eq_InVGH_presat} that $\mathbb{B}_{\mathcal{I}}$ wasn't $\delta$-cc in $W$.  It follows by \eqref{eq_YetAnotherEquiv} that $\mathbb{B}_{\bar{\mathcal{I}}}$ isn't $\delta$-cc in $W[K]$.  

In summary:  $W[K]$ satisfies that $\delta = \kappa^{+(n+1)}$ and $\bar{\mathcal{I}}$ is a $\delta$-presaturated, non-$\delta$-saturated ideal on $\wp_\kappa(\kappa^{+n})$, where $\kappa = \aleph_{\omega+1}$.  By Corollary \ref{cor_PresatNonsatGivesAnswer}, $W[K]$ satisfies that $(n, \aleph_{\omega+1}, \bar{\mathcal{I}})$ is a counterexample for Question \ref{q_Foreman}.

\section{Proof of Theorem \ref{thm_AnswerPotentAxioms}}\label{sec_Potent}

In this section we use both parts of Theorem \ref{thm_KillSaturation} to prove Theorem \ref{thm_AnswerPotentAxioms}, providing a negative answer to Foreman's Question \ref{q_Foreman_Potent}.   First we need a theorem implicit in Kunen~\cite{MR495118}.\footnote{Kunen proved how to do this for ideals on $\omega_1$, and Laver observed that it can be generalized to larger cardinals; see \cite{MattHandbook}.}  Foreman gives a stronger version as Theorem 7.70 in \cite{MattHandbook}.

\begin{theorem}[Kunen, Laver]\label{thm_Kunen}
If $\mu < \kappa$ are regular, $\kappa$ is a huge cardinal, and GCH holds, then there is a ${<}\mu$-closed forcing extension $W$ that satisfies:
\begin{itemize}
 \item GCH
 \item $\kappa = \mu^+$
 \item There is a $\kappa$-complete, normal, fine, $\kappa^+$-presaturated ideal $\mathcal{J} \subset \wp ([\kappa^+]^{\kappa})$.
 \end{itemize}
\end{theorem}

The forcing $\mathbb B_{\mathcal J}$ is in fact of the form ``$\delta$-cc followed by $\delta$-closed"---and thus $\delta = \kappa^+$-presaturated by Fact \ref{fact_ProperImpliesPresat}.

Fix a universe $W$ as in the conclusion of Theorem \ref{thm_Kunen}.\footnote{To answer Question \ref{q_Foreman_Potent} we just take $\mu = \omega_1$.}  Let $\mathbb{P}:=\mathbb{P}^W(\mu,\kappa)$ from Definition \ref{def_TodorcPoset}.  Let $G$ be $(W,\mathbb{P})$-generic.  By part \ref{item_NoSatIdealsKappa} of Theorem \ref{thm_KillSaturation}, there are no $\kappa$-complete, $\kappa^+$-saturated ideals on $\kappa$ in $W[G]$.   Let $\bar{\mathcal{J}}$ be the ideal generated by $\mathcal{J}$ in $W[G]$.  By part \ref{item_PresPresatKillSat} of Theorem \ref{thm_KillSaturation}---taking $\delta = \kappa^{+W}$---there is some $S \in \bar{\mathcal{I}}^+$ such that $\bar{\mathcal{I}}^+ \restriction S$ is $\kappa^+$-presaturated.  This completes the proof of Theorem \ref{thm_AnswerPotentAxioms}.

\section{Foreman's questions on regular embeddings}
\label{sec_Regularity}

In this section we answer Question \ref{q_Foreman_Regularity} part \ref{item_RegQuestYes} affirmatively and part \ref{item_RegQuestNo} negatively.  We also prove several theorems demonstrating the close relationship between strong properness and part \ref{item_RegQuestNo} of Question \ref{q_Foreman_Regularity}.

\begin{lemma}[Answering Question \ref{q_Foreman_Regularity} part \ref{item_RegQuestYes}]
\label{yes}
Suppose that $\mathcal J$ is a normal ideal on $Z \subset \wp (X )$, that $\mathbb P$ is a $|X |^+$-cc partial ordering, and that $\bar{\mathcal J}$ is a $|X |^+$-saturated normal ideal in $V^{\mathbb P}$. Then the map $\mathrm{id} : \wp (Z)/\mathcal J \to \mathbb P * \wp(Z)/\bar{\mathcal J}$ given by $[A]_{\mathcal J} \mapsto (1,[\check A]_{\bar{\mathcal J}})$
 is a regular embedding.
\end{lemma}

\begin{proof}
The map $\mathrm{id}$ is clearly an order and antichain preserving map since $\Vdash_{\mathbb P} \check{\mathcal J} \subseteq \bar{\mathcal J}$, so $\mathcal J$ is saturated.
Let $G * H$ be $\mathbb P * \wp(Z)/\bar{\mathcal J}$-generic over $V$.  Since $\bar{\mathcal J}$ is saturated in $V[G]$, there is an elementary embedding $j : V[G] \to N \subset V[G * H]$, where $N$ is transitive.  We claim that for all $A \in \wp(Z)^V$,
\begin{equation}
A \in \mathcal J \text{ iff } \Vdash_{\mathbb P * \mathbb B_{\bar{\mathcal J}}} j[X] \notin j(A)
\end{equation}
If $A \in \mathcal J$, then $\Vdash_{\mathbb P} A \in \bar{\mathcal J}$, and since $\bar{\mathcal J}$ is forced to be normal, it is forced that $[\mathrm{id}] = j[X]  \notin j(A)$.  If $A \in  \mathcal J^+$, then $\Vdash_{\mathbb P} A \in \bar{\mathcal J}^+$, and thus there is an extension in which $[\mathrm{id}] = j[X] \in j(A)$.

Since $\mathcal{J}$ is saturated, every antichain has size at most $|X|$.  Let $\langle [A_x] : x \in X \rangle$ be a maximal antichain in $\mathbb B_{\mathcal J}$.  Then $\nabla_{x\in X} A_x \in \text{Dual}(\mathcal J)$, so it is forced that $j[X] \in j(\nabla_{x \in X} A_x)$.  In other words, it is forced that for some $x \in X$, $[\mathrm{id}] \in j(A_x)$.  If $\langle (1,[\check A_x]_{\bar{\mathcal J}}) : x \in X \rangle$ were not a maximal antichain in $\mathbb P * \wp(Z)/\bar{\mathcal J}$, then there would be some $(p,[\dot B])$ incompatible with all $(1,[\check A_x])$.  If $G* H$ is generic with $(p,[\dot B]) \in G * H$, then we would have a generic embedding $j$ such that $[\mathrm{id}] \notin j(A_x)$ for any $x \in X$, a contradiction.
\end{proof}

The remainder of this section is related to part \ref{item_RegQuestNo} of Question \ref{q_Foreman_Regularity}.
\begin{definition}
Suppose $\mathbb P$ is a partial order, $\mathcal J$ is a normal ideal, and $\theta$ is a sufficiently large regular cardinal.  Let $\mathcal{F}$ be the conditional club filter relative to $\mathcal{J}$ (see Section \ref{sec_Prelims}).   We will say that $\mathbb P$ is \textbf{$\boldsymbol{\mathcal J}$-strongly proper} if
 \[ \{ M \prec H_\theta :  \mathbb{P} \text{ is strongly proper with respect to } M \} \in \mathcal{F} \]
\end{definition}

We remark that if $\mathcal{J} = \text{NS}_{\omega_1}$, then $\mathcal{J}$-strong properness is equivalent to the usual strong properness (see Section \ref{sec_ForcingPrelims}).  This is because for $\theta > 2^{\omega_1}$, every countable $M \prec H_\theta$ is $\text{NS}_{\omega_1}$-good, and so the conditional club filter relative to $\text{NS}_{\omega_1}$ is simply the club filter on the collection of \emph{all} countable $M \prec H_{\Omega(\text{NS}_{\omega_1})}$.  

\begin{theorem}\label{thm_CC_Reg_StrProp}
Suppose that $\mathcal J$ is a normal precipitous ideal on $Z \subseteq \wp (X )$, and $\mathbb P$ is a partial ordering with $|\mathbb{P}| \le |X|$.  The following are equivalent:
\begin{enumerate}
  \item\label{item_J_str_proper_regchar} $\mathbb P$ is $\mathcal J$-strongly proper.

  \item\label{item_id_regchar} For a dense set of $(p,A) \in  \mathbb P  \times \mathbb{B}_{\mathcal{J}}$, there is a $\mathbb{P}$-name $\dot{S}$ for an element of $\dot{\bar{\mathcal{J}}}^+$ such that:
  \begin{enumerate}
   \item\label{subitem_IdReg} The map $(q, T) \mapsto (q, \check{T})$ is a regular embedding from 
  \[
  \mathbb P \restriction p  \times \mathbb{B}_{\mathcal{J} \restriction A} \to \mathbb P \restriction p  *  \dot{\mathbb{B}}_{\dot{\bar{\mathcal{J}}} \restriction \dot{S}}.
  \]
  \item\label{subitem_ExtraReq} There exists a $\mathbb{B}_{\mathcal{J}}$-name $\dot{m}$ 
  such that 
  the triple $(A,\dot{S},\dot{m})$ satisfies the assumptions of Theorem \ref{thm_DualityThm}.\footnote{I.e.\ $A \in \mathcal{J}^+$; $\dot{m}$ is a $\mathbb{B}_{\mathcal{J}}$-name for a condition in $j_{\dot{U}}(\mathbb{P} \restriction p)$ such that $q \mapsto \big( 1, j_{\dot{U}}(q) \big)$ is a regular embedding from $\mathbb{P} \restriction p \to \mathbb{B}_{\mathcal{J}} * j_{\dot{U}}(\mathbb{P} \restriction p) \restriction \dot{m}$; there is an $f \in V$ such that $[A]_{\mathcal{J}} \Vdash \dot{m} = [\check{f}]_{\dot{U}}$; and $\dot{S}$ is the $\mathbb{P}$-name for $\{ z \in A \ : \ f(z) \in \dot{G} \}$ where $\dot{G}$ is the $\mathbb{P}$-name for its generic object.} 
  \end{enumerate}

\end{enumerate} 

\textbf{Moreover:}  if $\mathbb{P}$ is $\kappa$-cc and $\mathcal J$ is $\kappa$-complete, then part \ref{subitem_ExtraReq} is redundant; i.e. part \ref{subitem_ExtraReq} already follows from part \ref{subitem_IdReg}.
\end{theorem}

\begin{proof}
\ref{item_J_str_proper_regchar} $\implies$ \ref{item_id_regchar}:  Assume $\mathbb{P}$ is $\mathcal{J}$-strongly proper, and let $(p_0,A_0) \in  \mathbb P  \times \mathbb{B}_{\mathcal{J}}$ be arbitrary.

\begin{nonGlobalClaim}\label{clm_MainProductForcing}
There is some $\mathbb B_{\mathcal J}$-name $\dot m$ such that $1 \Vdash \dot m \leq j_{\dot U}(p_0)$, and
\[
(1,\dot m) \Vdash_{\mathbb{B}_{\mathcal{J}} * j_{\dot{U}}(\mathbb{P})} \ j_{\dot{U}}^{-1}[\dot{G}'] \times \dot{U} \ \text{ is generic over } V \text{ for } \mathbb{P} \times \mathbb{B}_{\mathcal{J}}
\]
\end{nonGlobalClaim}
\begin{proof}
(of Claim \ref{clm_MainProductForcing}) Fix a $\theta >> |X|$ and an $H \prec (H_\theta, \in, X,Z,\mathcal{J},\mathbb{P})$ such that $X,\mathbb{P} \subset H$ and $|H|=|X|$; this is possible by the assumption that $|\mathbb{P}| \le |X|$.  Let $U$ be $(V,\mathbb{B}_{\mathcal{J}})$-generic and $j_U: V \to N_U$ the generic ultrapower.  Since $|H|=|X|=|\text{sprt}(\mathcal{J})|$ and $\mathcal{J}$ is normal, $j_U[H] \in N_U$.  Also notice that since $\mathbb{P} \subset H$, $j_U[H] \cap j_U(\mathbb{P}) = j_U[\mathbb{P}]$.  If $D \in H \cap \text{Dual}(\mathcal{J})$ then genericity of $U$ implies $j_U[X] \in j_U(D)$, and since $X \subset H$ it follows that $j_U[X] = j_U[H] \cap j_U(X)$.  Thus $N_U \models j_U[H]$ is $j_U(\mathcal{J})$-good; so by our assumption, $N_U$ models that $j_U(\mathbb{P})$ is strongly proper with respect to $j_U[H]$.    By Corollary \ref{cor_Absolute}, this holds from the point of view of $V[U]$ as well.  Work in $V[U]$.  Let $m$ be any $\big(j_U[H], j_U(\mathbb{P}) \big)$-strong master condition below $j_U(p_0)$.  Then if $G'$ is generic over $V[U]$ for $j_U(\mathbb{P})$ with $m \in G'$, then $G' \cap j_U[H]= G' \cap j_U[\mathbb{P}]$ is generic over $V[U]$ for $j_U[\mathbb{P}]$.  Since $j_U$ is an isomorphism from $\mathbb{P} \to j_U[\mathbb{P}]$ it follows that $j_U^{-1}[G']$ is $(V[U], \mathbb{P})$-generic; so by the Product Lemma, $j_U^{-1}[G'] \times U$ is $(V,\mathbb{P} \times \mathbb{B}_{\mathcal{J}})$-generic.  Letting $\dot{m}$ be a $\mathbb{B}_{\mathcal{J}}$-name for this $m$, this completes the proof of the claim.
\end{proof}

For some $p \leq p_0$, it is forced by $\mathbb B_{\mathcal J}$ that for all $q \leq p$, $j_{\dot U}(q)$ is compatible with $\dot m$, since otherwise the set of $q$ for which this fails is dense below $p_0$, contradicting Claim \ref{clm_MainProductForcing}.  Therefore we have that $q \mapsto \big( 1, j_{\dot{U}}(q) \big)$ is an embedding from $\mathbb{P} \restriction p$ into $\mathbb{B}_{\mathcal{J}} * j_{\dot{U}}(\mathbb{P} \restriction p) \restriction \dot{m}$.  It is regular because, by Claim \ref{clm_MainProductForcing}, $\mathbb{B}_{\mathcal{J}} * j_{\dot{U}}(\mathbb{P} \restriction p) \restriction \dot{m}$ forces the inverse image of the generic under this map to be $\mathbb P$-generic over $V$ (which is one of the several characterizations of regular embeddings in Section \ref{sec_ForcingPrelims}).  Pick $A \in (\mathcal{J}\restriction A_0)^+$ that decides a representation of $\dot{m}$---i.e.\ $[A]_{\mathcal{J}}$ forces $\dot{m} = [\check{f}]_{\dot{U}}$ for some $f \in V$---and define $\dot{S}$ to be the $\mathbb{P}$-name for $\{ z \in A \ : \ f(z) \in \dot{G} \}$.  Then $A$, $\dot{m}$, and $\dot{S}$ are as in the assumptions of Theorem \ref{thm_DualityThm}.  Let $\phi$ be the isomorphism from the conclusion of Theorem \ref{thm_DualityThm}.  Now assume $G * \bar{U}$ is an arbitrary $(V,\mathbb{P} \restriction p * \dot{\mathbb{B}}_{\dot{\bar{\mathcal{J}}}} \restriction \dot{S})$-generic filter, and let 
\begin{equation}\label{eq_DefU}
U:= \bar{U} \cap \wp^V(Z)
\end{equation}
To finish the \ref{item_J_str_proper_regchar} $\implies$ \ref{item_id_regchar} direction of the proof, we need to show that $G \times U$ is generic over $V$ for $\mathbb{P} \times \mathbb{B}_{\mathcal{J}}$.  If $W*G'$ is the $\mathbb{B}_{\mathcal{J}}* j_{\dot{U}}(\mathbb{P})$-generic obtained by transferring $G*\bar{U}$ via $\phi$, then Theorem \ref{thm_DualityThm} implies that $W = \wp^V(Z) \cap \bar{U}$; so by \eqref{eq_DefU} it follows that $U = W$.  Again by Theorem \ref{thm_DualityThm}, $j_{\bar{U}}$ extends $j_U$, and $G = j_U^{-1}[G']$.  By Claim  \ref{clm_MainProductForcing}, 
\[
V[U*G'] \models G \times U \text{ is } (V,\mathbb{P} \times \mathbb{B}_{\mathcal{J}}) \text{-generic}
\] 
which completes the \ref{item_J_str_proper_regchar} $\implies$ \ref{item_id_regchar} direction of the proof.

\ref{item_id_regchar} $\implies$ \ref{item_J_str_proper_regchar}:  
Suppose for a contradiction that $\mathbb{P}$ is not $\mathcal{J}$-strongly proper.  Then, if $H_\theta$ is the support of the conditional club filter relative to $\mathcal{J}$, there is a stationary set $B$ of $\mathcal J$-good models such that for all $M \in B$, there is $p \in M$ which cannot be extended to an $(M,\mathbb P)$-strong master condition.  
By pressing-down, there is a stationary $B' \subseteq B$ and a $p_0 \in \mathbb P$ such that for all $M \in B'$, $p_0 \in M$ and $p_0$ cannot be extended to an $(M,\mathbb P)$-strong master condition.  This implies that $\mathbb P \restriction p_0$ is not $\mathcal J$-strongly proper.

Fix a wellorder $\Delta$ of $H_\theta$, and let
\[
\mathfrak{A}:= (H_\theta, \in, \Delta, X,Z,\mathbb{P},\mathcal{J})
\]
Let $T$ be the projection of $B'$ to $Z$, $T := \{ M \cap X : M \in B' \}$, and note $T \in \mathcal J^+$.\footnote{To see this, since $B'$ is stationary there is an $N \in B'$ with $N \prec (H_\theta, \in, T, \mathcal{J})$.  If $T$ were in $\mathcal{J}$, then by $\mathcal{J}$-goodness of $N$, $N \cap X$ would be in the complement of $T$, which is a contradiction to the definition of $T$ and the fact that $N \in B'$.}  We claim that for every $M \prec \mathfrak A$ such that $p_0 \in M$ and $M \cap X \in T$, $p_0$ cannot be extended to an $(M,\mathbb P)$-strong master condition.  This is because for such $M$, there is $N \in B'$ with $N \cap X = M \cap X$, and thus $N \cap \mathbb P = M \cap \mathbb P$.  Since $p_0$ cannot be extended to an $(N,\mathbb P)$-strong master condition, the claim follows from Corollary \ref{lem_DiffModelsSameInt}.

Fix an $H \prec \mathfrak{A}$ with $X \subset H$ and $|H|=|X|$; notice that $\mathbb{P} \subset H$ too.  Let $(p,A,\dot S)$ be a witness to (\ref{item_id_regchar}), with $p \leq p_0$ and $A \subseteq T$.
Let $U$ be $(V,\mathbb{B}_{\mathcal J})$-generic with $A \in U$, and $j_U: V \to N_U$ the generic ultrapower.  Since $A \in U$ we have $j_U[X] \in j_U(A)$, and because $X \subset H$ it follows that $j_U[H] \cap j_U(X) = j_U[X] \in j_U(A) \subseteq j_U(T)$.   And $j_U[H]$ is easily seen to be $j_U(\mathcal{J})$-good, and elementary in $j_U(\mathfrak{A})$.  
By elementarity,
\begin{equation}\label{eq_N_U_no_SMC}
\begin{split}
N_U \models j_U(p) \text{ cannot be extended to a }  \big( j_U[H], j_U(\mathbb{P}) \big) \text{-}\\
\text{strong master condition.}  
\end{split}
\end{equation}

Let $G$ be generic over $V[U]$ for $\mathbb{P} \restriction p$.  Then by the Product Lemma, $G \times U$ is generic over $V$ for $\mathbb{P} \restriction p \times \mathbb{B}_{\mathcal{J} \restriction A}$.  By assumption about the regularity of the identity map, we can force over $V[G \times U]$ with the quotient
\[
\frac{\mathbb{P} \restriction p * \dot{\mathbb{B}}_{\dot{\bar{\mathcal{J}}}} \restriction \dot{S}}{ G \times U}
\]
Let $G*\bar{U}$ be generic for this quotient; since this is the quotient by the identity map we have 
\begin{equation}\label{eq_U_is_Inters}
U = \bar{U} \cap \wp^V(Z).  
\end{equation}

Assumption \ref{subitem_ExtraReq} ensures that Theorem \ref{thm_DualityThm} is applicable.  Let $\phi$ be the isomorphism from Theorem \ref{thm_DualityThm}, and let $W*G'$ be the generic for $\wp(Z)/\mathcal{J} * j_{\dot{U}}(\mathbb{P})$ obtained by transferring $G*\bar{U}$ via $\phi$.  By Theorem \ref{thm_DualityThm} we have $W= \bar{U} \cap \wp^V(Z)$; so by \eqref{eq_U_is_Inters} it follows that $U = W$.  Theorem \ref{thm_DualityThm} further tells us that $\text{ult}\big( V[G], \bar{U} \big)$ is of the form $N_U[G']$, $j_{\bar{U}}$ extends $j_U$, and $j_U[G] \subseteq G'$; in particular $j_U(p) \in G'$.

Because $G \times U$ was chosen to be generic over $V$ for $\mathbb{P} \times \mathbb{B}_{\mathcal{J}}$, the Product Lemma implies that $G$ is $(V[U], \mathbb{P})$-generic and hence $(N_U, \mathbb{P})$-generic.  Since $j_U \restriction \mathbb{P}$ is an isomorphism from $\mathbb{P}$ to $j_U[\mathbb{P}]$, this is equivalent to saying that 
\begin{equation}
G' \cap j_U[\mathbb{P}]  = G' \cap \big( j_U[H] \cap j_U(\mathbb{P}) \big) \text{ is } \big( N_U, j_U[H] \cap j_U(\mathbb{P}) \big) \text{-generic}
\end{equation}

Since $j_U(p) \in G'$, then by the Forcing Theorem there is some $p' \in G'$ with $p' \le j_U(p)$, such that
\begin{equation}
p' \Vdash^{N_U}_{j_U(\mathbb{P})} \ \ \dot{G}' \cap j_U[H]   \text{ is } \big( N_U, j_U[H] \cap  j_U(\mathbb{P}) \big) \text{-generic}
\end{equation}
where $\dot{G}'$ is the $j_U(\mathbb{P})$-name for its generic object.    But this implies that $p'$ is a $\big( j_U[H], j_U(\mathbb{P}) \big)$-strong master condition, and since $p' \le j_U(p)$ this contradicts \eqref{eq_N_U_no_SMC}.

It remains to justify the ``moreover" part of the Theorem.  Suppose $\mathbb{P}$ is $\kappa$-cc and $\mathcal J$ is $\kappa$-complete.  Then Fact \ref{fact_DualityKappaCC} implies that the triple $(A,\check A,j_{\dot U}(p))$ satisfies the assumptions of Theorem \ref{thm_DualityThm}.  The regularity of the map $(p, T) \mapsto (p, \check{T})$ implies that $p \Vdash \check A \in \mathrm{Dual}(\bar{\mathcal J} \restriction \dot S)$; in other words, $p \Vdash \dot S \leq_{\bar{\mathcal J}} \check A$.
\end{proof}


If $T$ is a tree of uncountable regular height, the \textbf{forcing to specialize $\boldsymbol{T}$ with finite conditions} refers to the set of finite functions $f: T \to \omega$ such that whenever $t <_T t'$ are both in $\text{dom}(f)$, then $f(t) \ne f(t')$, ordered by reverse inclusion.
\begin{lemma}\label{lem_SpecNotStrProp}
Suppose $\kappa$ is regular, $\mathbb P$ is a partial order preserving the regularity of $\kappa$, $\dot T$ is a $\mathbb P$-name for a $\kappa$-Aronszajn tree, and $\dot{\mathbb Q}$ is a $\mathbb P$-name for the forcing to specialize $\dot T$ with finite conditions.  Then $\mathbb{P} * \dot{\mathbb Q}$ is not $\mathcal{J}$-strongly proper, for any normal ideal $\mathcal{J}$ of uniform completeness $\kappa$.
\end{lemma}
\begin{proof}
It suffices to prove that if $M \prec (H_\theta, \in, \mathbb{P}, \dot T)$ and $M \cap \kappa \in \kappa$, then no condition in $\mathbb{P} * \dot{\mathbb Q}$ can force that $M \cap (\dot G * \dot H)$ is generic for $M \cap (\mathbb{P} * \dot{\mathbb Q})$ over $V$.  Let $G \subset \mathbb P$ be generic over $V$.  If $M \cap G$ is not $(V,M \cap \mathbb P)$-generic, we are done.  Otherwise, $\delta_M := M \cap \kappa = M[G] \cap \kappa$.  Since the levels of $T$ have size $< \kappa$, $\{ s \in T  : \ \text{Lev}_T(s) < \delta_M \} = T \cap M[G]$.   Let $t$ be a node of $T$ at level $\delta_M$, and let $B = \{ s \in T \ : s <_T t \}$.

Let $H \subset \mathbb Q$ be generic over $V[G]$, and suppose for a contradiction that $H \cap M[G]$ is $(M[G] \cap \mathbb Q)$-generic over $V[G]$.  A simple density argument shows that for all $n \in \omega$, there is $s \in B$ and $q \in H$ such that $q(s) = n$.  However $q(s) \not= q(t)$ for any $q \in H$ and $s \in B$.
\end{proof}

\begin{corollary}\label{cor_AnswerPartB}[providing negative answer to Question \ref{q_Foreman_Regularity} part \ref{item_RegQuestNo}]
Suppose there is a saturated normal ideal $\mathcal J$ on $\kappa = \mu^+$.  Then there is a $\kappa$-cc partial order $\mathbb P$ that preserves the saturation of $\mathcal J$ and is such that there do \textbf{not} exist $A, \dot S$ such that the map $(p,[T]_{\mathcal J}) \mapsto (p,[\check{T}]_{\bar{\mathcal J}})$ is a regular embedding from
\[ \mathbb P \times \wp (Z)/(\mathcal J \restriction A) \to \mathbb P * \wp(Z)/(\bar{\mathcal J} \restriction \dot S) .\]
\end{corollary}
\begin{proof}
If we force with $\mathrm{Col}(\omega,\mu)$, then $\kappa$ becomes $\aleph_1$, and by \cite{MR1940513}, there exists a special $\kappa$-Aronszajn tree $T$ in the extension.  Let $\dot T$ be a name for such a tree, and let $\mathbb P = \mathrm{Col}(\omega,\mu) * \dot{\mathbb Q}$, where $\dot{\mathbb Q}$ is the forcing to (redundantly) specialize $\dot T$ with finite conditions.  Since $\dot T$ is forced to be special, $\mathbb{P}$ is $\kappa$-cc in any outer model where $\kappa$ is still a regular cardinal (see \cite{MR1940513} Lemma 16.18).  Fact \ref{fact_DualityKappaCC} implies that 
\begin{equation}\label{eq_EquivSpecial}
\mathbb P * \wp(\kappa)/\bar{\mathcal J} \sim \wp(\kappa)/ \mathcal J * j(\mathbb P)
\end{equation}

If $j: V \to_G N$ is a generic ultrapower by some generic $G \subset \wp(Z)/\mathcal{J}$ and $g \subset \mathrm{Col}(\omega,\mu)$ is generic over $V[G]$, then we can extend $j$ to $\bar j : V[g] \to N[g]$.  Because $\bar j(T)$ is special in $N[g]$, this is upwards absolute to $V[G][g]$.  Also since $\mathcal{J}$ is saturated, $j(\kappa)$ is a regular cardinal in $V[G][g]$.  In particular, $V[G][g]$ models that $j(T)$ is a $j(\kappa)$-Aronszajn tree, and hence that the forcing to specialize it---which is correctly computed by $N[g]$, i.e.\ the poset $\bar j(\mathbb{Q})$---is $j(\kappa)$-cc.  Together with \eqref{eq_EquivSpecial} this implies that $\mathbb P$ forces $\bar{\mathcal J}$ to be saturated.

Notice that $|\mathbb{P}| = \kappa = \text{sprt}(\mathcal{J})$.  If there were some $A \in \mathcal{J}^+$ and $\dot{S}$ as in the statement of the current corollary, then by Theorem \ref{thm_CC_Reg_StrProp}, $\mathbb P$ would be $\mathcal{J} \restriction A$-strongly proper,
which would contradict Lemma \ref{lem_SpecNotStrProp}. 
\end{proof}


\section{Questions}\label{sec_Questions}

In light of our negative solution to Question \ref{q_Foreman_Potent}, the following seems a natural attempt to salvage the original question: 

\begin{question}\label{q_Potent_Revised}
Suppose there is a normal, fine, $\omega_2$-complete, $\omega_3$-presaturated ideal on $[\omega_3]^{\omega_2}$.  Must there exist a presaturated ideal on $\omega_2$?
\end{question}

We remark that the model we used to give a negative answer to Question \ref{q_Foreman_Potent} does not give a negative answer to Question \ref{q_Potent_Revised}, because the model in the conclusion of Theorem \ref{thm_Kunen} also has a saturated ideal on $\kappa$ which generates a presaturated ideal in the extension by $\mathbb{P}(\mu,\kappa)$.   We also remark that if all instances of \emph{presaturated} are replaced by \emph{saturated} in Question \ref{q_Potent_Revised}, then the answer is yes;\footnote{Suppose $\mathcal{J}$ is a normal, fine, $\omega_2$-complete, $\omega_3$-saturated ideal on $[\omega_3]^{\omega_2}$.  Let $\mathcal{I}$ be its projection to an ideal on $\omega_2$, and $e: \mathbb{B}_{\mathcal{I}} \to \mathbb{B}_{\mathcal{J}}$ the induced boolean embedding.  Since $e$ is $\perp$-preserving and $\mathbb{B}_{\mathcal{J}}$ is $\omega_3$-cc, then $\mathbb{B}_{\mathcal{I}}$ is $\omega_3$-cc as well.  } however it is not known if the resulting hypotheses of that revised question are even consistent, so we include that question as well:

\begin{question}\label{q_SaturatedHuge}
Fix $n \ge 1$.  Can there exist a normal, fine, $\omega_n$-complete, $\omega_{n+1}$-saturated ideal on $[\omega_{n+1}]^{\omega_n}$?
\end{question}

We note that if Question \ref{q_SaturatedHuge} has a positive answer for $n = 2$, this would yield a positive answer to Question 31 of Foreman~\cite{MattHandbook} regarding saturated ideals on $\omega_2$ which are indestructible by $\omega_2$-cc forcing.\footnote{Via the proof of Corollary 27 of Foreman~\cite{MR819932}.}

Theorem \ref{thm_KillSaturation} generalized Baumgartner-Taylor~\cite{MR654852} by showing that under certain cardinal arithmetic assumptions, if $\kappa$ is the successor of a regular cardinal then there is a cardinal-preserving poset which:
\begin{enumerate}
 \item preserves presaturation (below some positive set) of \textbf{all} presaturated normal ideals of completeness $\kappa$, yet kills their saturation (if they had it to begin with); and
 \item forces that there are no saturated ideals on $\kappa$ in the extension.
\end{enumerate} 

Combined with Theorem \ref{thm_Presat_implies_Catch}, this allows one to start, for example, with an \emph{arbitrary} saturated ideal on a successor of a regular cardinal and force to kill its saturation but preserve its presaturation below some condition.

The following questions ask if Theorem \ref{thm_KillSaturation} can be further generalized in several ways:
\begin{question}
Assume GCH and fix regular uncountable cardinals $\kappa < \lambda$.  Assume $\kappa$ is the successor of a regular cardinal.  Is there a cardinal-preserving poset which preserves presaturation of any presaturated normal ideal on $\wp^*_\kappa(\lambda)$, yet forces that there are no saturated ideals on $\wp^*_\kappa(\lambda)$ in the extension?  
\end{question}

\begin{question}
Can Theorem \ref{thm_KillSaturation}---at least of part \ref{item_PresPresatKillSat} of that theorem---be generalized to cases where $\kappa$ is the successor of a singular cardinal?
\end{question}

\begin{question}
Can Theorem \ref{thm_KillSaturation}---at least of part \ref{item_PresPresatKillSat} of that theorem---be generalized to cases where $\kappa$ is inaccessible?  

\end{question}

The following forcing-theoretic question seems to be open, and a positive answer would allow for the removal of the assumption $2^{|Z|}<\delta^{+\omega}$ from the hypotheses of part \ref{item_PresPresatKillSat} of Theorem \ref{thm_KillSaturation}:
\begin{question}\label{q_2stepPresat}
Suppose $\delta$ is regular and $\mathbb{P}*\dot{\mathbb{Q}}$ is a 2-step iteration of $\delta$-presaturated posets.  Must $\mathbb{P}*\dot{\mathbb{Q}}$ be $\delta$-presaturated?
\end{question}
Note that by Fact \ref{fact_DeltaPlusOmega}, any counterexample would have to have size at least $\delta^{+\omega}$, and by the proof of Fact \ref{fact_ProperImpliesPresat} would have to kill the stationarity of the ground model's $\wp^*_\delta(\theta)$ for some $\theta  \ge \delta$.

\end{document}